\documentclass[a4paper]{amsart}
\usepackage{url}
\usepackage[all]{xy}
\usepackage[mathscr]{eucal}
\usepackage{amsmath,amssymb,amsfonts}
\newtheorem{theorem}{Theorem}[section]
\newtheorem{lemma}[theorem]{Lemma}
\newtheorem{corollary}[theorem]{Corollary}
\newtheorem{example}[theorem]{Example}

\newtheorem{proposition}[theorem]{Proposition}
\newtheorem{remark}[theorem]{Remark}

\usepackage{stackengine} 
\usepackage{xcolor}
\usepackage[all]{xy}

\title[The polycyclic inverse monoids]{The polycyclic inverse monoids\\ and\\ the Thompson groups\\ revisited}

\author{Mark V. Lawson}
\address{Mark V. Lawson, Department of Mathematics
and the
Maxwell Institute for Mathematical Sciences, 
Heriot-Watt University,
Riccarton,
Edinburgh EH14 4AS, 
UNITED KINGDOM}
\email{m.v.lawson@hw.ac.uk}

\begin{document} 

\begin{abstract} 
We revisit our construction of the Thompson groups from the polycyclic inverse monoids in the light of new research.
Specifically, we prove that the Thompson group $G_{n,1}$ is the group of units of a Boolean inverse monoid $C_{n}$
called the Cuntz inverse monoid. 
This inverse monoid is proved to be the tight completion of the polycyclic inverse monoid $P_{n}$.
The \'etale topological groupoid associated with $C_{n}$ under non-commutative Stone duality is the usual groupoid
associated with the corresponding Cuntz $C^{\ast}$-algebra. 
We then show that the group $G_{n,1}$ is also the group of automorphisms of a specific $n$-ary Cantor algebra:
this $n$-ary Cantor algebra is constructed first as the monoid of total maps of a restriction semigroup \`a la Statman and then in terms of labelled trees \`a la Higman. 
\end{abstract}

\maketitle

\section{Introduction}

We denote the polycyclic inverse monoid on $n$ generators by $P_{n}$ where $n$ is a finite natural number greater than or equal to 2.
The class of  polycyclic inverse monoids is one of the first interesting classes of inverse semigroups that arise.
They were introduced in \cite{NP} and you can find an elementary account of their theory in my book \cite[Section~9.3]{Lawson1998}.\footnote{The reader is cautioned that there I work with suffixes rather than prefixes.}
They are the syntactic monoids of the bracketing languages which are the prototypes of all
context-free languages and they arise as the monoid of actions of the pushdown stack of a pushdown automaton;
see Perrot's paper \cite{Perrot1971} and his thesis \cite{Perrot1972}.
It is worth noting that Perrot was led to introduce self-similar group actions via his study of the polycyclic inverse monoids \cite{Lawson2008}.
Motivated by the calculations in \cite{BJ}, special kinds of representations of the polycyclic monoids, called {\em strong representations}, 
are studied in \cite{Lawson2009, JL2012,HW2018}.
This leads to the introduction of what we call the {\em gauge inverse submonoid} of a polycyclic monoid.
Both the polycyclic inverse monoids and the gauge inverse monoids are used tacitly  in Cuntz's paper \cite{Cuntz}
and it is Cuntz's work that is a distant ancestor of our own.
Special inverse submonoids of the polycyclic inverse monoids are described in greater generality in \cite{MS}.
See also \cite{Lawson2007b} for further references.

Suffice it to say: the polycyclic inverse monoids have a proven track record.
The goal of this paper is to return to the work I carried out in \cite{Lawson2007b} but apply to it ideas developed in \cite{LV2020a, AAM}.
Specifically, we shall show how to construct the classical Thompson group $G_{n,1}$ as the group 
of units of the Boolean inverse monoid $C_{n}$,  called the {\em Cuntz inverse monoid}, constructed from the polycyclic inverse monoid $P_{n}$
using ideas from \cite{Lenz}.
This will set the results of \cite{Lawson2007b} in a more modern context.
See also \cite{LS}.
Recall that our earlier paper was itself reformulating part of \cite{Birget}
which in turn was developing ideas to be found in \cite{Scott1984, Scott1984b, Scott1992}.
In the remainder of this section, we shall recall the definition of the polycyclic inverse monoids as they appear in this paper.
They are most naturally defined in terms of free monoids, so we begin there.

Let $A_{n}$ be a finite alphabet with $n$ elements, where we assume that $n \geq 2$.
The {\em free monoid} on $A_{n}$ is denoted by $A_{n}^{\ast}$.
Its elements are called {\em strings} with  the identity element being the empty string $\varepsilon$.
The {\em length} of the string $x$ is denoted by $\mid x \mid$.
If $x = yz$, where $x,y,z$ are finite strings, we say that $y$ is a {\em prefix} of $x$.
If $x$ is a finite string then $\mathsf{Pref}(x)$ denotes the set of all prefixes of $x$.
For any subset $X$ of $A_{n}^{\ast}$ and any string $a \in A_{n}^{\ast}$,
define the set $a^{-1}X$ by the condition that $y \in a^{-1}X$  if and only if $ay \in X$.
For further properties of free monoids and all proofs, see \cite{Lallement}.

We shall be interested in right ideals in the monoid $A_{n}^{\ast}$.
Recall that a subset $R \subseteq A_{n}^{\ast}$ is a {\em right ideal}
if $r \in R$ and $x \in A_{n}^{\ast}$ implies that $rx \in R$.
Right ideals always have the form $XA_{n}^{\ast}$,where $X$ is any subset of $A_{n}^{\ast}$.
If this set is finite, then the right ideal is said to be {\em finitely generated};
if it is a singleton set, then the right ideal is said to be {\em principal}.
The first important property of the principal right ideals of $A_{n}^{\ast}$ 
is that the intersection of any two of them is either empty or itself a principal right ideal.
In fact, $xA_{n}^{\ast} \cap yA_{n}^{\ast}$ is equal to $xA_{n}^{\ast}$ if $y$ is a prefix of $x$;
it is equal to $yA_{n}^{\ast}$ if $x$ is a prefix of $y$;
and it is empty if neither of $x$ or $y$ is a prefix of the other.
If $xA_{n}^{\ast} \cap yA_{n}^{\ast} = \varnothing$ we say that $x$ and $y$ are 
{\em prefix incomparable}
whereas if this set is non-empty we say that $x$ and $y$ are {\em prefix comparable}.
Using the terminology from the theory of higher rank graphs,
we can say that free monoids are {\em singly aligned};
see \cite[Definition~20.1]{Exel}.

A finite subset of a free monoid is called a {\em prefix code} if any two distinct elements are prefix incomparable.
It is convenient to allow the empty set $\varnothing$ to be a prefix code.
A prefix code is said to be a {\em maximal prefix code} if every element of the free monoid
is prefix comparable with some element of the prefix code.
Prefix codes will play an important r\^ole in this paper.
We call $\{\varepsilon\}$ the {\em trivial maximal prefix code}.
The set $A_{n}$ is itself a maximal prefix code we call a {\em caret}.
See \cite{BP} for the theory pf prefix codes.
The following result is well-known but is basic in what follows.

\begin{lemma}\label{lem:colson} Let $X$ be any non-empty finite subset of $A_{n}^{\ast}$.
Then there is a prefix code $X' \subseteq X$ such that $X'A_{n}^{\ast} = XA_{n}^{\ast}$.
\end{lemma}
\begin{proof} Define a binary relation $\preceq$ on $A_{n}^{\ast}$ by $x \preceq y$ if $y$ is a prefix of $x$.
This is in fact a partial order.
For any finite subset $X$, denote by $\mbox{max}(X)$ the set of $\preceq$-maximal elements of $X$.
This is a prefix code.
We claim that $XA_{n}^{\ast} = \mbox{max}(X)A_{n}^{\ast}$.
Only one direction needs proving.
Let $x \in X$.
If $x$ is $\preceq$-maximal there is nothing to prove so assume it is not.
Then $x \preceq x'$ for some $\preceq$-maximal element $x'$.
Thus $x = x'u$ for some finite string $u$.
This completes the proof.
\end{proof}

Free monoids deal with finite strings but
we shall also need to work with some kinds of infinite strings.
Let $A_{n}$ be our finite alphabet where $n \geq 2$.
We denote by $A_{n}^{\omega}$ the set of all right-infinite strings over $A_{n}$.
This set can be endowed with a topology that turns it into the Cantor space
where the open sets are the subsets of the form $XA_{n}^{\omega}$ where $X$ is a set of finite strings.
The sets of the form $XA_{n}^{\omega}$ are clopen precisely when $X$ is finite.
See \cite{PP} for details and proofs.
Recall that a {\em Boolean space} is a compact, Hausdorff space with a basis of clopen sets.
The Cantor space is such a Boolean space.
The following is well-known.
A proof can be deduced from \cite[Lemma~4.3]{Lawson2007b}.

\begin{proposition}\label{prop:mpc} Let $Z$ be a prefix code over the alphabet $A_{n}$.
Then it is a maximal prefix code if and only if $ZA_{n}^{\omega} = A_{n}^{\omega}$.
\end{proposition}

We can now begin to define the polycyclic inverse monoids.
We shall be interested in certain kinds of functions between right ideals of free monoids.
A function $\theta \colon R_{1} \rightarrow R_{2}$ between two right ideals of a monoid $S$ is called a {\em morphism}
if $\theta (rs) = \theta (r)s$ for all $r \in R_{1}$ and $s \in S$.
Morphisms are therefore the analogues of the right module homomorphisms in module theory.
As usual, if $\alpha$ is a bijective morphism then $\alpha^{-1}$ is also a bijective morphism.
In the first instance, we shall be interested in bijective morphisms between principal right ideals of free monoids.
However, observe that $xA_{n}^{\ast} = yA_{n}^{\ast}$ if and only if $x = y$.
We may therefore label a principal right ideal by the string that generates it.
Given strings $x,y \in A_{n}^{\ast}$ we may define a function $xy^{-1} \colon yA_{n}^{\ast} \rightarrow xA_{n}^{\ast}$
given by $yu \mapsto xu$.
This map is a bijective morphism.
Observe that $\varepsilon \varepsilon^{-1}$ is the identity morphism which we shall usually just write as $1$.\\

\noindent
{\bf Definition. }The set of all bijective morphisms between the principal right ideals of the free monoid $A_{n}^{\ast}$ 
together with the empty partial function $0$ is denoted by $P_{n}$.
It is, in fact, an inverse monoid under composition of partial functions and it is this which is
called the {\em polycyclic monoid on $n$ generators}.\\

Observe that the non-zero idempotents of the polycyclic inverse monoid are the elements of the form $xx^{-1}$,
the inverse of $xy^{-1}$ is $yx^{-1}$,
and the natural partial order (see below) is given by $xy^{-1} \leq uv^{-1}$ if and only if $(x,y) = (u,v)p$ for some finite string $p$.
It follows that the polycyclic inverse monoids are $E^{\ast}$-unitary (see below);
this implies that they are $\wedge$-semigroups \cite[Remark~2.3]{Lenz} (see below).
In this paper, we shall need a little inverse semigroup theory since this forms the setting for our work.

We refer the reader to \cite{Lawson1998} for background on inverse semigroups;
we recall some key definitions here.
An {\em inverse semigroup} is a semigroup in which for each element $a$ there is a unique element, denoted by $a^{-1}$,
such that $a = aa^{-1}a$ and $a^{-1} = a^{-1}aa^{-1}$.
The set of idempotents in $S$ is denoted by $\mathsf{E}(S)$.
It is called the {\em semilattice of idempotents} of $S$.
If $S$ is an inverse {\em monoid} its group of units is denoted by $\mathsf{U}(S)$.
Define $\mathbf{d}(a) = a^{-1}a$ and $\mathbf{r}(a) = aa^{-1}$.
Define the {\em natural partial order} on $S$ by $a \leq b$ if and only if $a = ba^{-1}a$.
It can be proved that with respect to this order, an inverse semigroup is partially ordered.
In addition, $a \leq b$ implies that $a^{-1} \leq b^{-1}$.
An inverse semigroup is called {\em $E$-unitary} if $e \leq a$, where $e$ is an idempotent, implies that $a$ is an idempotent.
An inverse semigroup with zero is called  {\em $E^{\ast}$-unitary} if $e \leq a$, where $e$ is a non-zero idempotent, implies that $a$ is an idempotent.
Define the {\em compatibility relation} $a \sim b$ precisely when $a^{-1}b$ and $ab^{-1}$ are both idempotents.
If $a \sim b$ we say that $a$ and $b$ are {\em compatible}.
A non-empty subset $X$ of an inverse semigroup is said to be {\em compatible}
if each pair of elements of $X$ is compatible.
Observe that if $a,b \leq c$ then $a \sim b$.
It follows that $a \sim b$ is a necessary condition for $a$ and $b$ to have a join $a \vee b$
with respect to the natural partial order.
Idempotents $e$ and $f$ are said to be {\em orthogonal}, written $e \perp f$, if $ef = 0$.
The elements $a$ and $b$ are said to be {\em orthogonal}, also written $a \perp b$, if $\mathbf{d}(a) \perp \mathbf{d}(b)$
and $\mathbf{r}(a) \perp \mathbf{r}(b)$.
Orthogonal elements are certainly compatible.
The join of orthogonal elements will also be called an {\em orthogonal join}.
The proof of the following is straightforward.

\begin{lemma}\label{lem:orthogonality}
In the polycyclic inverse monoid $P_{n}$, we have that $xx^{-1} \perp yy^{-1}$
if and only if $x$ and $y$ are prefix incomparable.
\end{lemma}

By Lemma~\ref{lem:orthogonality}, 
the subset $\{x_{1}, \ldots, x_{m} \}$ of $A_{n}^{\ast}$ is a prefix code if and only if   $\{x_{1}x_{1}^{-1}, \ldots, x_{m}x_{m}^{-1} \}$
is an orthogonal subset of $P_{n}$.
Now, $x = yp$ if and only if $xx^{-1} \leq yy^{-1}$.
The proof of the following is now straightforward.

The following lemma will be useful to us.
See \cite[Lemma 1.4.11 and 1.4.12]{Lawson1998} for a proof.
We may paraphrase it by saying that compatible elements have a meet which is algebraically defined.

\begin{lemma}\label{lem:compatible-meets} Let $S$ be an inverse semigroup.
If $a \sim b$ then $a \wedge b$ exists and is equal to $ab^{-1}b = ba^{-1}a$ and we have that
$\mathbf{d} (a \wedge b) = \mathbf{d}(a)\mathbf{d}(b)$ and $\mathbf{r} (a \wedge b) = \mathbf{r}(a)\mathbf{r}(b)$.
\end{lemma} 

An inverse monoid is said to be {\em distributive} if each pair of compatible elements has a join
and multiplication distributes over such joins.
A {\em morphism} of distributive inverse semigroups maps compatible joins to compatible joins.
A distributive inverse monoid is said to be {\em Boolean}
if its idempotents form a Boolean algebra under the natural partial order.
An inverse semigroup is called a {\em $\wedge$-semigroup} or a {\em meet semigroup} if each pair of elements has a meet.
A pseudogroup is an inverse semigroup in which every compatible subset has a join and multiplication distributes
over such joins.
The following lemma tells us how meets and joins interact in a distributive inverse semigroup.
For a proof see \cite[Lemma 2.5]{Lawson2016}.

\begin{lemma}\label{lem:meets-joins}
Let $S$ be a distributive inverse semigroup.
Suppose that $\bigvee_{i} a_{i}$ is defined and $b \wedge \left( \bigvee_{i} a_{i}  \right)$ is defined.
Then all meets $b \wedge a_{i}$ are defined, the join $\bigvee_{i} b \wedge a_{i}$ is defined and
$$b \wedge \left( \bigvee_{i} a_{i}  \right)
=
\bigvee_{i} b \wedge a_{i}.$$
\end{lemma}

The minimum group congruence $\sigma$ on an inverse semigroup is defined by $a \, \sigma \, b$ if and only if
there exists an element $c$ such that $c \leq a,b$.
See \cite[Section~2.4]{Lawson1998} for more information on the minimum group congruence.
An inverse semigroup is said to be {\em $F$-inverse} if each $\sigma$-class contains a maximum element.
Every $F$-inverse semigroup is $E$-unitary.

A congruence $\rho$ on a semigroup with zero $S$ is said to be {\em $0$-restricted} if $0\, \rho \, a$ implies that $a = 0$.
A congruence $\rho$ on an inverse semigroup is said to be {\em idempotent-pure} if $e \,\rho \,a$, where $e$ is an idempotent, implies that $a$ is an idempotent.
If $a\, \rho \, b$, where $\rho$ is idempotent-pure, then $a \sim b$.
This observation will be important in this paper.

Let $(P,\leq)$ be a poset.
For each non-empty subset $A \subseteq P$, define $A^{\downarrow}$ to be all those elements of $P$ below some element of $A$
and define $A^{\uparrow}$ to be all those elements above some element of $A$.
If $A = A^{\downarrow}$ we say that $A$ is an {\em order ideal}.
If $A = A^{\uparrow}$ we say that $A$ is {\em closed upwards}.
If $A = \{a\}$ we write $a^{\downarrow}$ instead of $\{a\}^{\downarrow}$ and 
$a^{\uparrow}$ instead of $\{a\}^{\uparrow}$.

The notion of a `tight cover' will play an important r\^ole throughout this paper.
Let $S$ be an inverse semigroup.
Consider a subset $\{a_{1}, \ldots, a_{m} \} \subseteq a^{\downarrow}$.
Observe that if $x \leq a$ then $x$ is compatible with any $a_{i}$.
It follows that all meets $x \wedge a_{i}$ exist by Lemma~\ref{lem:compatible-meets}.
We say that $\{a_{1}, \ldots, a_{m} \}$ is a {\em tight cover} of $a$ if $0 < x \leq a$ implies that $x \wedge a_{i} \neq 0$ for some $i$. 
If $\{b\}$ is such that $b \leq a$ and is a tight cover then we write $b \leq_{e} a$ and say that
$b$ is {\em essential} in $a$.
The proof of the following is routine using Lemma~\ref{lem:meets-joins}.

\begin{lemma}\label{lem:nanaimo} Let $S$ be a distributive inverse semigroup.
Then $\{a_{1}, \ldots, a_{m}\}$ is a tight cover of $a$ if and only if
$\bigvee_{i=1}^{m} a_{i} \leq_{e} a$.
\end{lemma}

Let $S$ be an inverse semigroup.
A subset $A \subseteq S$ is a {\em filter} if $A = A^{\uparrow}$
and if $a,b \in A$ there exists $c \in A$ such that $c \leq a,b$.
It is {\em proper} if it does not contain $0$.
The proper filter $A$ is {\em tight} if $a \in A$ and  $\{a_{1}, \ldots, a_{m}\}$ 
a tight cover of $a$ implies that $a_{i} \in A$ for some $i$.
A maximal proper filter is called an {\em ultrafilter}.
If $S$ is a distributive inverse semigroup a proper filter $A$ is said to be {\em prime}
if $a \vee b \in A$ implies that $a \in A$ or $b \in A$.

\begin{lemma}\label{lem:tate} Let $\{x_{1}x_{1}^{-1}, \ldots, x_{p}x_{p}^{-1} \}$ be a tight cover of $xx^{-1}$.
Put $X = \{x_{1},\ldots, x_{p} \}$ and $X' = \mbox{\rm max}(X)$.
Then $\{yy^{-1} \colon y \in X' \}$ is a tight cover of $xx^{-1}$.
\end{lemma}
\begin{proof} Let $zz^{-1} \leq xx^{-1}$.
Thus $z = xu$ for some finite string $u$.
By definition, there exists $w$ such that $w = zv = x_{i}t$ for some finite strings $v$ and $t$.
If $x_{i} \in X'$ then we are done, otherwise there exists $y \in X'$ such that $x_{i} = ys$ for some finite string $s$.
It follows that $w = zv = yst$.
Thus $ww^{-1} \leq zz^{-1}, yy^{-1}$.
This proves that  $\{yy^{-1} \colon y \in X' \}$ is a tight cover of $xx^{-1}$.
\end{proof}

Observe that the set $\{yy^{-1} \colon y \in X' \}$ above is an orthogonal set by Lemma~\ref{lem:orthogonality}.

\begin{lemma}\label{lem:zorn} Let $S$ be an inverse semigroup with zero.
\begin{enumerate}
\item Then every non-zero element is contained in an ultrafilter. 
\item Every proper filter is contained in an ultrafilter.
\end{enumerate}
\end{lemma}
\begin{proof} (1) Let $s \neq 0$.
Then $s^{\uparrow}$ is a proper filter containing $s$.
The set of all proper filters that contain $s$ is non-empty and is partially ordered by set inclusion.
The union of every chain of proper filters is a proper filter.
Thus by Zorn's lemma the set of all proper filters containing $s$ has a maximal element.
It follows that every non-zero element of $S$ is contained in an ultrafilter.
(2) The proof is similar to that in (1).
\end{proof}

\begin{lemma}\label{lem:filter-inclusions} In an inverse semigroup, every ultrafilter is a tight filter, 
and in a distributive inverse semigroup every tight filter is a prime filter.
\end{lemma}
\begin{proof} We prove first that every ultrafilter is a tight filter.
By \cite[Proposition 2.13]{Lawson2010},
it is enough to prove the result for ultrafilters in meet semilattices with zero.
Let $F$ be such an ultrafilter in a meet semilattice with zero and suppose that 
$e \in F$ and $\{e_{1},\ldots, e_{m}\}$ is a cover of $e$.
Suppose that none of $e_{1},\ldots, e_{m}$ belongs to $F$.
Then for each $e_{i}$ there exists $f_{i} \in F$ such that $e_{i} \wedge f_{i} = 0$ by \cite[Lemma 12.3]{Exel}.
Put $f = \bigwedge_{i=1}^{m} f_{i}$.
Then $f$ is non-zero and belongs to $F$. 
Thus $e \wedge f$ is non-zero and belongs to $F$.
But $0 < e \wedge f \leq e$ and $(e \wedge f) \wedge e_{i} = 0$ for all $i$.
This contradicts the assumption that  $\{e_{1},\ldots, e_{m}\}$ is a cover of $e$.
The proof that every tight filter is a prime filter is immediate by Lemma~\ref{lem:nanaimo}.
\end{proof}

Let $S$ be an arbitrary inverse semigroup with zero.
If $A$ is an ultrafilter then $\mathbf{d}(A) = (A^{-1}A)^{\uparrow}$
is an ultrafilter.
Similarly, $\mathbf{r}(A) = (AA^{-1})^{\uparrow}$
is an ultrafilter.
Let $A$ and $B$ be ultrafilters.
If $\mathbf{d}(A) = \mathbf{r}(B)$ then $A \cdot B = (AB)^{\uparrow}$ is an ultrafilter.
The proofs of all the above results can be deduced from \cite{Lawson2010}.

\begin{lemma}\label{lem:groupoids} Let $S$ be an inverse semigroup with zero.
Then $\mathsf{G}(S)$, the set of all ultrafilters of $S$ equipped with the partial binary operation $\cdot$,
is a groupoid.
\end{lemma}

Let $V_{s}$ be the set of all ultrafilters that contain the element $s$.
Put $\tau = \{V_{s} \colon s \in S\}$.
By Lemma~\ref{lem:zorn}, we have that $s \neq 0$ implies that $V_{s} \neq \varnothing$.

\begin{lemma} With the above definition, the set
$\tau$ is a base for a topology on the set of all ultrafilters $\mathsf{G}(S)$ on $S$.
\end{lemma}
\begin{proof} We refer the reader to \cite[Theorem 5.3]{Willard}
for the definition of a base of open sets.
Observe that $\mathsf{G}(S) = \bigcup_{a \in S} V_{a}$.
Let $A \in V_{a} \cap V_{b}$.
Then the ultrafilter $A$ contains $a$ and $b$ and so it contains an element $c \in A$ such that
$c \leq a,b$.
Thus $A \in V_{c} \subseteq V_{a}, V_{b}$.
\end{proof}

See \cite{Resende} for information on \'etale topological groupoids.

\begin{theorem}\label{them:etale-groupoid} Let $S$ be an inverse semigroup with zero.
Then $\mathsf{G}(S)$ is an \'etale topological groupoid when equipped with the topology with open base $\tau$.
\end{theorem}
\begin{proof} By Lemma~\ref{lem:groupoids}, we know that $\mathsf{G}(S)$ is a groupoid.
We adapt results to be found in \cite{Lawson2010} to this more general setting.
We prove first that $\mathsf{G}(S)$ is a topological groupoid.
Thus we have to prove that the inversion map and the multiplication map are both continuous maps.
The fact that the inversion map is continuous follows from the fact that $V_{a^{-1}} = (V_{a})^{-1}$.
Continuity of the multiplication map follows by the same argument as in the proof of step~3 of \cite[Proposition~2.22]{Lawson2010}.
It remains to show that it is \'etale.
We prove that the map $\mathbf{d} \colon V_{a} \rightarrow V_{a^{-1}a}$ is a homeomorphism.
The fact that this map is a bijection follows from the properties of ultrafilters.
It is injective because if $A$ and $B$ are two ultrafilters containing the element $a$ such that $\mathbf{d}(A) = \mathbf{d}(B)$
then $A = B$.
It is surjective because if $F \in  V_{a^{-1}a}$  then $A = (aF)^{\uparrow}$ is an ultrafilter containing $a$
such that $\mathbf{d}(A) = F$.
The map is continuous because inversion and multiplication are continuous.
It remains to show that this is an open map.
Let $V_{b} \subseteq V_{a}$.
Then it is routine to check that $V_{b^{-1}b} \subseteq V_{a^{-1}a}$;
alternatively, we can use the fact that $V_{a}V_{b} = V_{ab}$ which again follows by basic properties of ultrafilters.
\end{proof}

In Sections 2, 3 and 4 we revisit the material from \cite{Lawson2007, Lawson2007b} in the light of the recent generalizations to be found in \cite{LV2020a, AAM};
the main theorems we prove are Theorem~\ref{them:one}, Theorem~\ref{them:tight-completion} and Theorem~\ref{them:associated-groupoid}.
In Section~5, we branch out into a new direction by developing the theory of Cantor algebras from the perspective of semigroup theory;
the main theorems we prove are Theorem~\ref{them:higman-one}, Theorem~\ref{them:free}, Theorem~\ref{them:higman-two} and Theorem~\ref{them:higman-forms}.\\

\noindent
{\bf Acknowledgements }I am grateful to Phil Scott (Ottawa) for directing me to Statman's work and numerous Skype discussions.\\

\section{The Thompson groups via the polycyclic inverse monoids}

In this section, we shall construct the group associated with a polycyclic inverse monoid.

\subsection{The distributive inverse monoid $D_{n}$}

We begin with a general construction.
Let $S$ be an arbitrary inverse semigroup.
Define $\mathsf{C}(S)$ to be the set of all compatible order ideals of $S$.
Then under subset multiplication, $\mathsf{C}(S)$ is a pseudogroup and the map $\iota \colon S \rightarrow \mathsf{C}(S)$,
given by $s \mapsto s^{\downarrow}$, is a homomorphism universal for homomorphisms to the category of pseudogroups;
see \cite[Theorem~1.4.23]{Lawson1998}.
Now, define $\mathsf{D}(S)$ to be the finite elements of $\mathsf{C}(S)$ \cite[Section~3.2]{LL}.
These are the finitely generated compatible order ideals of $S$.
This is a distributive inverse semigroup.
Let $\iota \colon S \rightarrow \mathsf{D}(S)$ be the restriction of the above map.
This is universal to the category of distributive inverse semigroups \cite[Theorem 4.30]{LL}, \cite[Proposition 2.5]{Lawson2012}, \cite{LL2}, \cite{Lawson2020}.
Define $S$ to satisfy the {\em weak meet condition} if the intersection of any two principal order ideals
is finitely generated as an order ideal;
this notion is due to Steinberg \cite{Steinberg2010}.
Then $S$ satisfies the weak meet condition if and only if $\mathsf{D}(S)$ is a $\wedge$-semigroup.

Observe that each element of $\mathsf{D}(S)$ is a finite join of elements in $S$,
and that if $a \in \{a_{1}, \ldots, a_{m} \}^{\downarrow}$, where $a,a_{1},\ldots, a_{m} \in S$,
then $a \leq a_{i}$ for some $i$.
This observation leads to the following abstract description of the distributive completion $\mathsf{D}(S)$ of the inverse semigroup $S$.

\begin{theorem}\label{them:abstract-dist-compl} Let $S$ be an inverse semigroup with zero.
Suppose that $S$ is an inverse subsemigroup of a distributive inverse semigroup $T$ such that the following two properties are satisfied:
\begin{enumerate}
\item Each element of $T$ is a finite join of elements from $S$.
\item If $a \leq \bigvee_{i=1}^{m} a_{i}$, where $a,a_{i} \in S$, then $a \leq a_{i}$ for some $i$. 
\end{enumerate}
Then $T$ is the distributive completion of $S$.
\end{theorem}
\begin{proof} Denote the embedding of $S$ in $T$ by $\iota$.
Let $\alpha \colon S \rightarrow D$ be any monoid homomorphism
to a distributive inverse monoid $D$.
We shall define a morphism $\beta \colon T \rightarrow D$.
Each element of $t \in T$ can be written as a join $t = \bigvee_{i=1}^{m} s_{i}$
of elements of $S$.
It follows that $\{s_{i} \colon 1 \leq i \leq m \}$ is a compatible subset of $S$.
Thus $\{\alpha (s_{i}) \colon 1 \leq i \leq m \}$ is a compatible subset of $D$.
We may therefore define
$$\beta \left(  \bigvee_{i=1}^{m} s_{i}   \right) = \bigvee_{i=1}^{m} \alpha (s_{i}).$$
However, we have to show that this is well-defined.
Suppose that
$$\bigvee_{i=1}^{p} a_{i} = \bigvee_{j=1}^{q} b_{j}$$
in $T$.
Then, for each $i$, we have that $a_{i} \leq \bigvee_{j=1}^{q} b_{j}$.
Thus for each $i$ there exists a $j$ such that  $a_{i} \leq b_{j}$
which is an inequality in $S$.
Thus $\alpha (a_{i}) \leq \alpha (b_{j})$.
It is now clear (also by symmetry)
that the function $\beta$ is well-defined.
By construction,  the map $\beta$ is a morphism of distributive inverse monoids
and it is clearly the unique such morphism satisfying $\beta \iota = \alpha$.
\end{proof}

Denote by $D_{n}$ the set of all bijective morphisms between the finitely generated right ideals of the free monoid $A_{n}^{\ast}$.

\begin{lemma}\label{lem:dist-inv}
The inverse monoid $D_{n}$ is a distributive inverse $\wedge$-monoid.
\end{lemma}
\begin{proof} Because free monoids are singly aligned, the intersection of two finitely generated right ideals
is a finitely generated right ideal.
If $\theta \colon XA_{n}^{\ast} \rightarrow YA_{n}^{\ast}$ is a bijective morphism
and $ZA_{n}^{\ast}$ is a finitely generated right ideal in $XA_{n}^{\ast}$ then
$\theta (ZA_{n}^{\ast})$ is a finitely generated right ideal in $YA_{n}^{\ast}$.
Clearly, $A_{n}^{\ast}$ is a finitely generated right ideal.
It follows that $D_{n}$ is an inverse monoid. 
The union of finitely generated right ideals is a finitely generated right ideal.
It is now easy to show that $D_{n}$ is a distributive inverse monoid.
It remains to show that it is a $\wedge$-monoid.
We shall use \cite{Leech}.
Let $\theta \colon XA_{n}^{\ast} \rightarrow YA_{n}^{\ast}$ be a bijective morphism.
Observe first that $\theta (X)A_{n}^{\ast} = YA_{n}^{\ast}$.
We can therefore assume in what follows that $Y = \theta (X)$.
We are interested in the elements $z \in XA_{n}^{\ast}$ such that $\theta (z) = z$.
By assumption, $z = xu$ where $x \in X$ and $u \in A_{n}^{\ast}$.
Thus $\theta (x)u = xu$.
By cancellation, we deduce that $\theta (x) = x$.
Let $X'$ be the set of all $x \in X$ such that $\theta (x) = x$.
Then the fixed point subset of $\theta$ is the set $X'A_{n}^{\ast}$.
The identity function on this set is therefore the largest idempotent in $D_{n}$
less than or equal to $\theta$.
It follows by \cite{Leech} that $D_{n}$ has all binary meets.
\end{proof}

Clearly, there is an embedding $\iota \colon P_{n} \rightarrow D_{n}$.
The way that $P_{n}$ sits inside $D_{n}$ will be crucial in what follows. 
The proofs of the following are straightforward.

\begin{proposition}\label{prop:skyr} In the distributive inverse monoid $D_{n}$ the following conditions hold:
\begin{enumerate}
\item Each element of $D_{n}$ is a finite join of elements from $P_{n}$.
\item If $a \leq \bigvee_{i=1}^{m} a_{i}$, where $a,a_{i} \in P_{n}$, then $a \leq a_{i}$ for some $i$. 
\item If $a \leq b$, where $a,b \in P_{n}$ and $a$ is a non-zero idempotent, then
$b$ is an idempotent. This is simply the fact that $P_{n}$ is $E^{\ast}$-unitary.
\end{enumerate}
\end{proposition}

By Theorem~\ref{them:abstract-dist-compl} and Proposition~\ref{prop:skyr},
the embedding of $P_{n}$ in $D_{n}$ has a universal property
which thereby puts the paper \cite{Lawson2007} into its proper context.

\begin{theorem}[Distributive completion]\label{them:dist-completion}The distributive inverse monoid $D_{n}$ is 
the distributive completion of the inverse monoid $P_{n}$.
\end{theorem}

The above theorem is important because it shows that $D_{n}$ is {\em algebraically} natural.

\begin{remark}\label{rem:normal-form}
{\em We can get a normal form for elements of $D_{n}$.
Let $\theta \colon XA_{n}^{\ast} \rightarrow YA_{n}^{\ast}$ be a bijective morphism
where $X$ and $Y$ are finite sets.
We can choose $X$ to be a prefix code by Lemma~\ref{lem:colson}.
Observe that $\theta (X)A_{n}^{\ast} = YA_{n}^{\ast}$
and so $Y$ can also be chosen to be a prefix code.
Suppose that $X$ and $Y$ are prefix codes and $\theta (X) = Y$.
Order the sets $X$ and $Y$ so that $\theta (x) = y$;
we shall make this ordering assumption throughout this paper and extend it to unions as well
(for convenience).
Then we obtain what we shall term a {\em symbol} $\left( \frac{X}{Y} \right)$.
All symbols $\left( \frac{\varnothing}{Y} \right)$ represent the empty partial bijection.
Observe that a symbol uniquely determines a bijective morphism between finitely generated right ideals since if $X$ and $X'$ are prefix codes we have that
$XA_{n}^{\ast} = X'A_{n}^{\ast}$ if and only if $X = X'$.}
\end{remark}

\begin{remark}{\em 
The above theorem implies the work of \cite{Lawson2007} as we now show.
Let $\theta \colon XA_{n}^{\ast} \rightarrow YA_{n}^{\ast}$ be a bijective morphism where $\theta (X) = Y$.
We can write $\theta = \bigvee_{i=1}^{m} y_{i}x_{i}^{-1}$ where $X = \{x_{1}, \ldots, x_{m}\}$ and  $Y = \{y_{1}, \ldots, y_{m}\}$
and where we can assume, without loss of generality, that $X$ and $Y$ are prefix codes by Remark~\ref{rem:normal-form}.
However, the set  $\{x,y\}$ is prefix incomparable if and only if the set $\{xx^{-1}, yy^{-1}\}$ is orthogonal by Lemma~\ref{lem:orthogonality}.
It follows that $\theta$ is an orthogonal join of elements in $P_{n}$.}
\end{remark}

The following result will be useful to us later when we come to consider the \'etale groupoid associated with $D_{n}$.
Recall that $\mathsf{G}(S)$ is the \'etale groupoid of ulrafilters of $S$ by Theorem~\ref{them:etale-groupoid}.

\begin{theorem}\label{them:isomorphism-of-etale-groupoids} 
Let $S$ be an inverse semigroup and let $T$ be its distributive completion.
Then $\mathsf{G}(S) \cong \mathsf{G}(T)$, an isomorphism of \'etale groupoids.
\end{theorem}
\begin{proof} We use Theorem~\ref{them:abstract-dist-compl}
and the fact that every ultrafilter is a prime filter by Lemma~\ref{lem:filter-inclusions}.
We first establish a bijection between 
$\mathsf{G}(S)$ and  $\mathsf{G}(T)$ as follows.
If $A$ is an ultrafilter in $S$ then $(A)^{\uparrow}$ is an ultrafilter in $T$;
if $A$ is an ultrafilter in $T$ then $A \cap S$ is an ultrafilter in $S$.
Observe that these operations are mutually inverse.
It is routine to check that this bijection establishes an isomorphism of groupoids.
As to the topologies, observe that in a distributive inverse semigroup such as $T$ we have that
$V_{a \vee b} = V_{a} \cup V_{b}$ since ultrafilters are prime filters.
Thus the topology in $T$ is really determined by those open sets of the form $V_{a}$ where $a \in S$.
\end{proof}

\subsection{The group associated with a polycyclic inverse monoid}

A non-zero idempotent $e$ in an inverse semigroup is said to be {\em essential}
if for every non-zero idempotent $f$ we have that $ef \neq 0$.

\begin{lemma}\label{lem:essential-idempotents}
Let $S$ be an inverse monoid.
Then $e$ is an essential idempotent if and only if $e \leq_{e} 1$.
\end{lemma}
\begin{proof} Suppose first that $e$ is an essential idempotent.
Since the idempotents form an order-ideal in an inverse semigroup with maximum element $1$ in the monoid case,
we may write $f \leq 1$ precisely when $f$ is an idempotent.
It follows that if $f$ is non-zero then $fe \leq 0$.
Thus $e \leq_{e} 1$.
The proof of the converse is now similar.
\end{proof}

Let $S$ be an inverse semigroup.
We say that $a \in S$ is {\em essential} if both idempotents $\mathbf{d}(a)$ and $\mathbf{r}(a)$ are essential.
Denote by $S^{e}$ the set of all essential elements of $S$.
Then $S^{e}$ is an inverse subsemigroup of $S$ and, crucially, it will not contain the zero element.
See \cite[Lemma~2.8]{Lawson2007b}.

We now look at the inverse semigroup $D_{n}^{e}$ consisting of all elements of $D_{n}$
whose domains and ranges are essential idempotents.
We can obtain a different description of the elements of $D_{n}^{e}$.
The following was proved as \cite[Lemma~A.1]{Birget}.

\begin{lemma}\label{lem:stolz} 
Every essential finitely generated right ideal of $A_{n}^{\ast}$ is generated by a finite maximal prefix code (and, conversely).
\end{lemma}

It follows that the elements of $D_{n}^{e}$ are the bijective morphisms between the essential finitely generated right ideals
--- therefore, by the above lemma, between the right ideals generated by the finite maximal prefix codes.
The following is the substance of \cite[Proposition 2.1]{Birget}.

\begin{proposition}\label{prop:finverse}
The inverse monoid $D_{n}^{e}$ is $F$-inverse.
\end{proposition}

Recall that $F$-inverse semigroups are also $E$-unitary.
We can now define the groups that interest us;
these are precisely the famous Thompson groups.\\

\noindent
{\bf Definition. }Let $n \geq 2$.
Define the group $G_{n,1}$ to be the group $D_{n}^{e}/\sigma$.
When $n = 2$ this is the Thompson group $V$.\\

\section{The structure of prefix codes}

The material in this section is nothing more than a reworking of classical theory.
We describe the structure of the maximal prefix codes in Theorem~\ref{them:max-pc}
and an important relation between prefix codes in Proposition~\ref{prop:cummings}.

The following is well-known but we give a proof anyway.

\begin{lemma}\label{lem:prefix-codes} Let $X_{1}, \ldots, X_{n}$ be prefix codes in the free monoid $A_{n}^{\ast}$.
Then $X = a_{1}X_{1} \cup \ldots \cup a_{n}X_{n}$ is a prefix code.
\end{lemma}
\begin{proof} Without loss of generality, we shall assume that our prefix codes are non-empty.
 For each $i$, suppose that $a_{i}x$ and $a_{i}y$ are prefix comparable where $x,y \in X_{i}$.
Then $a_{i}xu = a_{i}yv$ where $u$ and $v$ are some finite strings.
By left cancellation, we have that $x$ and $y$ are prefix comparable and so must be equal since $X_{i}$ is a prefix code.
Clearly, strings that begin with different letters cannot be prefix comparable.
It follows that $X$ is a prefix code.
\end{proof}

\begin{lemma}\label{lem:quotients-pc} Let $X$ be any prefix code.
Then, for each $i$, the set $a_{i}^{-1}X$ is a prefix code.
\end{lemma}
\begin{proof}
Let $x,y \in a_{i}^{-1}X$ be prefix comparable.
Then $xu = yv$ for some finite strings $u$ and $v$.
But $a_{i}x, a_{i}y \in X$.
However, $a_{i}xu = a_{i}yv$ which implies that $a_{i}x = a_{i}y$.
Thus by left cancellation, we have that $x = y$, as required.
\end{proof}

We call the process in part (1) below {\em a caret expansion}
and that in part (2), we call  {\em a caret reduction};
the proofs are routine.

\begin{lemma}\label{lem:carets} Let $X$ be a prefix code where $X \neq \varnothing$.
\begin{enumerate}
\item Let $x \in X$ be any element.
Then $X' = X \setminus \{ x\} \cup \{xa_{1}, \ldots, xa_{n} \}$ is also a prefix code.
If $X$ is a maximal prefix code then so too is $X'$.

\item Let $x \in X$ be such that $xA_{n} \subseteq X$.
Then $X'' = X \setminus xA_{n} \cup \{x\}$ is also a prefix code.
If $X$ is a maximal prefix code then so too is $X''$.
\end{enumerate}
\end{lemma}

We can use caret expansions and reductions to construct all maximal prefix codes.
The following is well-known but we give a proof for completeness.

\begin{theorem}\label{them:max-pc} Fix the alphabet $A_{n}$.
Then each maximal prefix code is either the trivial prefix code or is obtained from the trivial prefix code
by a finite sequence of caret expansions.
\end{theorem}
\begin{proof} Let $X$ be a non-trivial maximal prefix code.
We shall prove first that a sequence of caret reductions will transform it to the trivial maximal prefix code.
To kick off, we shall prove that there is an $x \in X$ such that $XA_{n} \subseteq X$.
We prove first that if every element of $X$ has length $1$ then $X = A_{n}$.
Let $a \in A_{n}$ be arbitrary.
Then since $X$ is a maximal prefix code, we have that $au = bv$ for some finite strings $u$ and $v$ and $b \in X$.
But then $a = b$, since they both have length $1$.
We have proved that $a \in X$ and so $X = A_{n}$.
Now we can deal with the general case.
We assume in what follows that $X$ contains strings of length $2$ or more.
Let $x$ be a string of maximal length in $X$;
by assumption, it has length at least 2.
Let $x = x'a$ where $a$ is a letter and $x'$ has length one less than that of $x$.
We claim that $x'A_{n} \subseteq X$.
Let $b \in A_{n}$ be arbitrary.
We shall prove that $x'b \in X$.
Since $X$ is a maximal prefix code we have that $x'bu = yv$ where $y \in X$ and $u$ and $v$ are finite strings.
It follows that $x'b$ and $y$ are prefix comparable.
The length of $y$ is at most the length of $x'b$ since $x'a$ has maximal length.
If they have the same length then $x'b = y$ and we have proved that $x'b \in X$.
Suppose that $y$ has length strictly less than $x'b$.
If it has length equal to $x'$ then $x' = y$ but then $y$ is a proper prefix of $x$ which contradicts the fact that $X$ is a prefix code.
If $y$ has a length strictly less than that of $x'$ then $y$ is a proper prefix of $x'$ and so a proper prefix of $x$.
Again, this contradicts the fact that $X$ is a prefix code.
It follows that $x'b \in X$ for all $b \in A_{n}$.
We have therefore proved that $xA_{n} \subseteq X$.
Thus $X' = X \setminus xA_{n} \cup \{ x\}$ is a maximal prefix code.

If $X'$ is trivial then we are done, otherwise we now repeat the above process.
At some point, we shall arrive at the trivial maximal prefix code.
Thus $X$ can be reduced to the trivial maximal prefix code by a sequence of caret reductions.
By reversing these operations, that is by carrying out a sequence of caret extensions, 
it follows that the maximal prefix code $X$ is obtained from the trivial maximal prefix code by a sequence of caret expansions.
\end{proof}

We say that a maximal prefix code $X$ in $A_{n}^{\ast}$ is {\em uniform} if all elements have the same length.
We say it has {\em height  $r$} if the length of the longest branch from leaf to root is $r$.

\begin{lemma}\label{lem:trump} 
Any two uniform maximal prefix codes of the same height over the alphabet $A_{n}$ are equal.
\end{lemma}
\begin{proof} Let $X$ and $Y$ be two uniform maximal prefix codes of height $r$.
Let $x \in X$.
Then $x$ has length $n^{r}$.
There exists $y \in Y$ such that $xu = yv$ for some finite strings $u$ and $v$.
But $y$ also has length $n^{r}$.
It follows that $x = y$ and so $x \in Y$.
By symmetry, we get that $X = Y$.
\end{proof}

The notion of tightness is intimately bound up with the structure of prefix codes, as we now show.
Because of Proposition~\ref{prop:skyr}, we can apply the following lemma to $D_{n}$.

\begin{lemma}\label{lem:tight-covers} Let $S$ be a distributive inverse monoid that contains  an inverse submonoid $T$ such that the following conditions hold:
\begin{enumerate}
\item Each element of $S$ is a finite join of elements from $T$.
\item If $a \leq \bigvee_{i=1}^{m} a_{i}$ where $a,a_{i} \in T$ then $a \leq a_{i}$ for some $i$. 
\item If $a \leq b$, where $a,b \in T$ and $a$ is a non-zero idempotent, then
$b$ is an idempotent.
\end{enumerate}
Then 
$\bigvee_{i=1}^{s} e_{i} \leq_{e} \bigvee_{j=1}^{t} f_{j}$, where $e_{i}, f_{j} \in T$, if and only if
$\bigvee_{i=1}^{s} e_{i} f_{j} \leq_{e} f_{j}$ for each $j$.
\end{lemma}

The proof of the following easily follows from the definitions.

\begin{lemma}\label{lem:mpc} In the inverse monoid $D_{n}$,
the orthogonal set $\{x_{1}x_{1}^{-1}, \ldots, x_{m}x_{m}^{-1} \}$ is a tight cover of $\varepsilon \varepsilon^{-1}$ if and only if
$\{x_{1}, \ldots, x_{m} \}$ is a maximal prefix code.
\end{lemma}

The following lemma is key.
The proof is routine. 

\begin{lemma}\label{lem:corona} In the inverse monoid $D_{n}$,
let $\{x_{1}x_{1}^{-1}, \ldots, x_{m}x_{m}^{-1}\} \subseteq (xx^{-1})^{\downarrow}$ 
where $x_{i} = xp_{i}$.
Then this is a tight cover if and only if 
 $\{p_{1}p_{1}^{-1}, \ldots, p_{m}p_{m}^{-1}\} \subseteq (1)^{\downarrow}$ 
is a tight cover.
\end{lemma} 

\begin{remark}\label{rem:lagavulin}
{\em By Lemma~\ref{lem:corona}, tight covers of arbitrary idempotents in $P_{n}$
are determined by tight covers of the identity $1$.
By Lemma~\ref{lem:tate}, we can assume that we are dealing with an orthogonal set
and by Lemma~\ref{lem:stolz}, this means that we are dealing with an underlying maximal prefix code.
By Theorem~\ref{them:max-pc}, the maximal prefix codes are determined by caret expansions.
This all means that the tight covers of idempotents $xx^{-1}$ may ultimately be traced back to the
maximal prefix code $A_{n}$.
Observe that $\{a_{1}a_{1}^{-1}, \ldots, a_{n}a_{n}^{-1}\}$ is a tight cover of the identity
and the prototype of all tight covers of idempotents.
In fact, by Lemma~\ref{lem:tight-covers}, it is the prototype for all tight covers in $P_{n}$.
These observations will be placed into their proper mathematical setting in Theorem~\ref{them:tight-completion}.}
\end{remark}

We may now prove the following result.

\begin{proposition}\label{prop:cummings}
We work in the inverse monoid $D_{n}$.
Let $X$ and $Y$ be prefix codes.
We write $1_{X}$ for the identity function defined on the set $XA_{n}^{\ast}$ and $1_{Y}$ for the identity function defined on $YA_{n}^{\ast}$.
Then $1_{X} \leq_{e} 1_{Y}$ if and only if the prefix code $X$ is obtained from the prefix code $Y$ by a series of caret expansions.
\end{proposition} 
\begin{proof} If $X$ is obtained from $Y$ as a result of one caret expansion then it is easy to check that  $1_{X} \leq_{e} 1_{Y}$.
We therefore need only prove the converse.
Let $X = \{x_{1}, \ldots, x_{s}\}$ and let $Y = \{y_{1},\ldots, y_{t} \}$.
Then, by assumption, $\bigvee_{i=1}^{s} x_{i}x_{i}^{-1} \leq_{e} \bigvee_{j=1}^{t} y_j y_{j}^{-1} $.
By Lemma~\ref{lem:tight-covers}, we have that
 $\bigvee_{i=1}^{s} x_{i}x_{i}^{-1}y_{j}y_{j}^{-1} \leq_{e} y_{j}y_{j}^{-1}$
 for each $j$.
 Suppose that 
 $x_{i}x_{i}^{-1}y_{j}y_{j}^{-1} \neq 0$ for some $i$.
 There are two possibilities.
 Case 1:
 if 
  $x_{i}x_{i}^{-1}y_{j}y_{j}^{-1} = x_{i}x_{i}^{-1}$
 then $x_{i}x_{i}^{-1} \leq y_{j}y_{j}^{-1}$.
 Case 2:
 if 
 $x_{i}x_{i}^{-1}y_{j}y_{j}^{-1} = y_{j}y_{j}^{-1}$
 then $y_{j}y_{j}^{-1} \leq x_{i}x_{i}^{-1}$.
 Let's look at the second case.
 If  $y_{j}y_{j}^{-1} \leq x_{i}x_{i}^{-1}$ then $x_{i}x_{i}^{-1} \leq y_{k}y_{k}^{-1}$ for some $k$.
 It follows that  $y_{j}y_{j}^{-1} = x_{i}x_{i}^{-1}$.
 Because of orthogonality we may therefore assume that only case 1 holds.
 We therefore have that 
 $\bigvee_{i=1}^{s'} x_{i}x_{i}^{-1} \leq_{e} y_{j}y_{j}^{-1}$
 where $x_{i} = y_{j}p_{i}$ where we have relabelled if necessary.
 Thus $\{x_{1}x_{1}^{-1}, \ldots, x_{s'}x_{s'}^{-1} \}$ is a tight cover of $y_{j}y_{j}^{-1}$
 by Lemma~\ref{lem:nanaimo}.
 By Lemma~\ref{lem:corona}, we have that
 $\{p_{1}p_{1}^{-1}, \ldots, p_{s'}p_{s'}^{-1} \}$ is a tight cover of $\varepsilon \varepsilon^{-1}$.
 Thus by Lemma~\ref{lem:mpc}, we have that $\{p_{1}, \ldots, p_{s'}\}$ is a maximal prefix code.
 We can now apply the structure theorem for maximal prefix codes described in Theorem~\ref{them:max-pc}
 to deduce that  $\{p_{1}, \ldots, p_{s'}\}$ is obtained from the trivial maximal prefix code by a sequence of prefix expansions.
 It follows that $\{x_{1}, \ldots, x_{s'} \}$ is obtained from $\{y_{j}\}$ by a sequence of caret expansions.
 The result now follows.
\end{proof}

\section{The tight completion}

In this section, we shall prove that the group defined in Section~2 is actually the group of units of a Boolean inverse monoid $C_{n}$
constructed from $D_{n}$. Our main tool is the Lenz congruence which is the subject of the next section.
We shall construct $C_{n}$ as $D_{n/\equiv}$.
This is defined in terms of one element being essential in another by virtue of Lemma~\ref{lem:nanaimo}.

\subsection{The Lenz congruence $\equiv$}

We begin by stating some results from \cite[Section~9]{LV2020a} in a slightly more general setting.
All proofs can be found in \cite[Section~4.1]{AAM}.
We recall first a definition due to Daniel Lenz \cite{Lenz}.
Let $S$ be an inverse semigroup with zero.
Define the relation $\equiv$ on $S$ as follows:
$s \equiv t$ if and only if for each $0 < x \leq s$ we have that $x^{\downarrow} \cap t^{\downarrow} \neq 0$ 
and for each $0 < y \leq t$ we have that $y^{\downarrow} \cap s^{\downarrow} \neq 0$.
Then $\equiv$ is a $0$-restricted congruence on $S$.
We call it the {\em Lenz congruence}.
We denote the $\equiv$-class of $a$ by $[a]$.
Clearly, if $a \leq_{e} b$ then $a \equiv b$ (see the Introduction for the meaning of this symbol).

\begin{lemma} \label{lem:needed}
In an inverse semigroup, we have that $a \leq_{e} b$ if and only if $\mathbf{d}(a) \leq_{e} \mathbf{d}(b)$.
\end{lemma}
\begin{proof} Suppose first that $a \leq_{e} b$.
We prove that  $\mathbf{d}(a) \leq_{e} \mathbf{d}(b)$.
Let $0 < e \leq \mathbf{d}(b)$.
Then $\mathbf{d}(be) = e$ and so $0 < be \leq b$.
It follows that $a \wedge be \neq 0$.
But $a \sim be$.
It follows that $\mathbf{d}(a) \wedge e \leq 0$.
Thus $\mathbf{d}(a) \leq_{e} \mathbf{d}(b)$, as claimed.
To prove the converse, suppose that 
$\mathbf{d}(a) \leq_{e} \mathbf{d}(b)$.
We prove that $a \leq_{e} b$.
Let $0 < x \leq b$.
Then $0 < \mathbf{d}(x) \leq \mathbf{d}(b)$.
It follows that $\mathbf{d}(x) \mathbf{d}(a) \neq 0$.
But $a \sim x$ and so $\mathbf{d} (a \wedge x) \neq 0$ giving $a \wedge x \neq 0$, as required.
\end{proof}

Let $\theta \colon S \rightarrow T$ be a homomorphism.
We say that it is {\em essential} if $a \leq_{e} b$ implies that $\theta (a) = \theta (b)$.
We define essential congruences in the obvious way.

\begin{lemma}\label{lem:important-tool} Let $S$ be an inverse semigroup.
If $\rho$ is any $0$-restricted, idempotent-pure essential congruence on $S$
then $a \, \rho \, b$ implies that $a \wedge b$ is defined and $(a \wedge b) \leq_{e} a,b$.
\end{lemma}
\begin{proof} Let $a \, \rho \, b$.
Since $\rho$ is idempotent-pure, we have that $a \sim b$.
Thus $a \wedge b$ is defined by Lemma~\ref{lem:compatible-meets}.
We prove that $(a \wedge b) \leq_{e} a$.
By assumption $a \wedge b = a \mathbf{d}(b) = b\mathbf{d}(a)$.
Observe that $a \, \rho \, b \mathbf{d}(a)$ and so $a \wedge b \neq 0$.
Also, $a \, \rho \, (a \wedge b)$.
Let $0 < x \leq a$.
It follows that $x \sim a \wedge b$.
Thus $x \wedge (a \wedge b) = (a \wedge b)\mathbf{d}(x)$.
We deduce that $x \rho b\mathbf{d}(a)\mathbf{d}(x)$.
Using the fact that $\rho$ is $0$-restricted we deduce that $(a \wedge b) \wedge x \neq 0$.
\end{proof}

\begin{lemma}\label{lem:new-characterization} Let $S$ be an inverse semigroup
in which $\equiv$ is idempotent-pure.
Then $a \equiv b$ if and only if there exists $c \leq_{e} a,b$.
\end{lemma}
\begin{proof} It is immediate that if there exists $c \leq_{e} a,b$ then $a \equiv b$.
The converse follows by  Lemma~\ref{lem:important-tool}.
\end{proof}

The above result enables us to connect essential idempotents with the Lenz congruence
by Lemma~\ref{lem:essential-idempotents}. 

The proof of the following is immediate by Lemma~\ref{lem:important-tool} and Lemma~\ref{lem:new-characterization}.

\begin{theorem}\label{them:uniqueness} 
Let $S$ be an inverse semigroup on which $\equiv$ is idempotent-pure.
Then $\equiv$ is the unique $0$-restricted, idempotent-pure essential congruence on $S$.
\end{theorem}

Theorem~\ref{them:uniqueness} goes some way to explain the true nature of the Lenz congruence,
at least in the context of this paper.

The following was proved as \cite[Lemma~9.12]{LV2020a}.

\begin{lemma}\label{lem:idpt-pure-dist} Let $S$ be a distributive inverse semigroup.
If $\rho$ is idempotent-pure then $S/\rho$ is a distributive inverse semigroup and the natural map from $S$ to $S/\rho$
is a morphism of distributive inverse semigroup.
If, in addition, $S$ is a $\wedge$-semigroup then $S/\rho$ is a $\wedge$-semigroup
and the morphism preserves meets.
\end{lemma}

The proof of the following is immediate by Lemma~\ref{lem:nanaimo}.

\begin{lemma}\label{lem:tight} Let $S$ be a distributive inverse semigroup.
Every tight filter is a prime filter.
\end{lemma}

By \cite[Proposition~5.10]{LL} and Lemma~\ref{lem:tight}, we have that:
every ultrafilter is a tight filter, and every tight filter is a prime filter.

\begin{lemma}\label{lem:mars} Let $S$ be an inverse semigroup with zero
in which $\equiv$ is idempotent-pure.
Let $X$ be a tight filter in $S$.
\begin{enumerate}
\item If $x \in X$ and $y \leq_{e} x$ then $y \in X$.
\item If $x \in X$ and $y \equiv x$ then $y \in X$.
\end{enumerate}
\end{lemma}
\begin{proof} (1) By definition, $\{y\}$ is a tight cover of $x$.
It follows that $y \in X$.
(2) This follows by (1) and Lemma~\ref{lem:new-characterization}.
\end{proof}

To prove that a distributive inverse semigroup is Boolean,
we have to prove, by \cite[Lemma~3.20]{LL}, that every prime filter is an ultrafilter.
By Lemma~\ref{lem:idpt-pure-dist}, if $S$ is distributive and $\equiv$ is idempotent-pure
then $S/\equiv$ is distributive.
The following theorem is now relevant.
It is proved as \cite[Theorem~4.12]{AAM}.

\begin{theorem}\label{them:seven} Let $S$ be a distributive inverse semigroup on which $\equiv$ is idem\-potent-pure. 
Then $S/\equiv$ is Boolean if and only if every tight filter in $S$ is an ultrafilter.
\end{theorem}

The following can easily be deduced from \cite{Lawson2010, Lawson2012}.
Let $A$ be a filter in the inverse semigroup $S$.
Define $\mathbf{d}(A) = (A^{-1}A)^{\uparrow}$.
Then $\mathbf{d}(A)$ is a filter in $S$ which is also an inverse subsemigroup. 
Furthermore, $A = (a \mathbf{d}(A))^{\uparrow}$ for any $a \in A$. 
Clearly, $0 \in A$ if and only if $0 \in \mathbf{d}(A)$.
Also, $\mathbf{d}(A) \cap \mathsf{E}(S)$ is a filter in $\mathsf{E}(S)$.
Lemma~\ref{lem:rioja} below shows that to check whether every tight filter is an ultrafilter
it is enough to restrict ones attention to the distributive lattice of idempotents.

\begin{lemma}\label{lem:rioja} Let $A$ be a filter in an inverse semigroup $S$.
Then $x \in \mathbf{d}(A)$ if and only if $a^{-1}a \leq x$ for some $a \in A$.
\end{lemma}

The following was proved as \cite[Proposition~4.14]{AAM}.

\begin{proposition}\label{prop:idpts-semigroup} Let $S$ be a distributive inverse semigroup.
Then every tight filter in $S$ is an ultrafilter in $S$
if and only if
every tight filter in $\mathsf{E}(S)$ is an ultrafilter in $\mathsf{E}(S)$.
\end{proposition}

We may summarize the results of this section in the following universal characterization.

\begin{theorem} Let $S$ be a distributive inverse semigroup in which every tight filter is an ultrafilter and $\equiv$ is 
idempotent-pure.
Put $T = S/\equiv$.
Then $T$ is a Boolean inverse semigroup and the natural map $\beta \colon S \rightarrow T$ 
is universal for all essential morphism $\theta \colon S \rightarrow B$ to Boolean inverse semigroups.
\end{theorem}
\begin{proof}
By Theorem~\ref{them:seven}, we have that $T$ is a Boolean inverse semigroup.
Let $\theta \colon S \rightarrow B$ be a homomorphism with the property that if $a \leq_{e} b$ then $\theta (a) = \theta (b)$.
We shall prove that there is a unique morphism $\phi \colon T \rightarrow B$ 
such that $\phi \beta = \theta$.
Let $a\, \equiv \, b$ in $S$.
Then by Lemma~\ref{lem:new-characterization}, there is $c \leq_{e} a, b$.
Thus, by assumption, $\theta (c) = \theta (a) = \theta (b)$.
In particular, $\theta (a) = \theta (b)$.
Denoting the $\equiv$-class containing the element $a$ by $[a]$ we may therefore unambiguously define
$\phi ([a]) = \theta (a)$.
It remains to prove that $\phi$ is a morphism.
Suppose that $[a] \sim [b]$ in $T$.
Then $a \sim b$ since $\equiv$ is idempotent-pure.
It follows that $\theta (a) \vee \theta (b)$ is defined.
We now use Lemma~\ref{lem:idpt-pure-dist}, to deduce that $[a \vee b] = [a] \vee [b]$.
The result now follows.
\end{proof}

We shall now reinterpret some results from \cite{Lenz}.
Let $S$ be an arbitrary inverse semigroup with zero.
Let $s \in S$ and recall that by $V_{s}$ we mean the set of all ultrafilters containing the element $s$.
Lemma~\ref{lem:dlenz} below is included for the sake of completeness.

\begin{lemma}\label{lem:dlenz} Let $S$ be an inverse semigroup with zero in which $\equiv$ is idempotent-pure.
Then $s \equiv t$ if and only if $V_{s} = V_{t}$.
\end{lemma}
\begin{proof} Suppose first that $s \equiv t$.
Let $A \in V_{s}$.
Thus $s \in A$ where $A$ is an ultrafilter.
But every ultrafilter is a tight filter by Lemma~\ref{lem:filter-inclusions}.
Thus by Lemma~\ref{lem:mars}, we have that $t \in A$ and so $A \in V_{t}$.
The result now follows by symmetry.
Now suppose that  $V_{s} = V_{t}$.
We prove that $s \equiv t$.
Let $0 < a \leq t$.
Let $a \in A$ where $A$ is an ultrafilter.
Then $A \in A_{t}$.
It follows that $s \in A$, by our assumption.
Now $a, s \in A$ and so there is a non-zero element $c \leq a,s$.
It now follows by symmetry that $s \equiv t$.
\end{proof}

The following will be useful to us later.

\begin{theorem}\label{them:isomorphism-etale-groupoids} 
Let $S$ be an inverse semigroup with zero on which the congruence $\equiv$ is idempotent-pure.
Then the \'etale topological groupoids $\mathsf{G}(S)$ and $\mathsf{G}(S/\equiv)$ are isomorphic.
\end{theorem}
\begin{proof} Put $T = S/\equiv$.
Denote by $\nu \colon S \rightarrow T$ the natural map associated with the congruence $\equiv$.
Observe first, that there is a bijection between the ultrafilters in $S$ and the ultrafilters in $T$.
We prove this now.
It is easy to check that if $A'$ is an ultrafilter in $T$, then $\nu^{-1}(A')$ is an ultrafilter in $S$,
and if $A$ is an ultrafilter in $S$, then $\nu (A)^{\uparrow}$ is an ultrafilter in $T$. 
Let $A$ and $B$ be ultrafilters in $S$ and suppose that $\nu (A)^{\uparrow} = \nu (B)^{\uparrow}$.
To prove that $A = B$ we use the fact that ultrafilters are tight filters together with Lemma~\ref{lem:filter-inclusions} and Lemma~\ref{lem:mars}.
Let $A'$ be an arbitrary ultrafilter in $T$.
Then $\nu^{-1}(A')$ is an ultrafilter in $S$.
Thus $\nu (\nu^{-1}(A'))^{\uparrow}$ is an ultrafilter in $T$.
Clearly, $A' \subseteq \nu (\nu^{-1}(A'))^{\uparrow}$.
But $A'$ is an ultrafilter.
Thus  
$A' = \nu (\nu^{-1}(A'))^{\uparrow}$.
Thus, we have established our bijection.
It is now routine to check that this bijection is an isomorphism of groupoids
and that it is a homeomorphism of topological groupoids.
\end{proof}

\subsection{The structure of the Lenz congruence on $D_{n}$}

We shall apply the results of the previous section to the distributive inverse semigroup $D_{n}$.
The following is key to doing this.

\begin{lemma}\label{lem:panda} Let $S$ be a distributive inverse semigroup.
Let $\mathcal{B}$ be a subset of $S$ having the following properties:
\begin{enumerate}
\item Each element of $S$ is a finite join of elements from $\mathcal{B}$.
\item If $a \leq \bigvee_{i=1}^{m} a_{i}$ where $a,a_{i} \in \mathcal{B}$ then $a \leq a_{i}$ for some $i$. 
\item If $a \leq b$, where $a,b \in \mathcal{B}$ and $a$ is a non-zero idempotent, then
$b$ is an idempotent.
\end{enumerate}
Then $\equiv$ is idempotent-pure on $S$.
\end{lemma}
\begin{proof} Suppose that $a \equiv e$ where $e$ is an idempotent.
Then we may write $a = \bigvee_{i=1}^{p} a_{i}$, where $a_{i} \in \mathcal{B}$ by property (1).
Thus $\left(    \bigvee_{i=1}^{p} a_{i}  \right) \equiv e$.
For each $i$, we have that $a_{i}$ is less than or equal to the lefthand side and is non-zero.
Thus, there is a non-zero element $z$ in $S$ such that $z \leq a_{i}, e$.
Now, $z$ will be a join of elements of $\mathcal{B}$ by property (1) again.
Thus we may assume that there exists  $0 \neq c \in \mathcal{B}$ where $c \leq a_{i}, e$.
But $c \leq e$ implies that $c$ is an idempotent and so an idempotent in $\mathcal{B}$.
Thus $a_{i}$ is an idempotent in $\mathcal{B}$ by property (3).
It follows that $a$ is itself an idempotent, as required.
\end{proof}

In the case of the distributive inverse monoid $D_{n}$ the inverse submonoid $P_{n}$ fulfills the r\^ole of $\mathcal{B}$ above
in Lemma~\ref{lem:panda} by Proposition~\ref{prop:skyr}.
We therefore have the following.

\begin{corollary}\label{cor:idpt-pure} 
The congruence $\equiv$ is idempotent-pure on $D_{n}$.
\end{corollary}

By Lemma~\ref{lem:dist-inv}, Corollary~\ref{cor:idpt-pure}, and Lemma~\ref{lem:idpt-pure-dist}, we therefore have the following.

\begin{lemma}\label{lem:skype}
The inverse monoid $D_{n}/\equiv$ is a distributive $\wedge$-monoid.
\end{lemma}

The proof of the following can be deduced using the proof of \cite[Lemma~9.7]{LV2020a}.

\begin{lemma}\label{lem:red-panda} Let $T$ be an inverse monoid in which $\equiv$ is idempotent-pure
and suppose that $T^{e}$ is $E$-unitary.
Then the group of units of $T/\equiv$ is isomorphic to $T^{e}/\sigma$.
\end{lemma}

We can give a much more concrete description of the congruence $\equiv$ defined on $D_{n}$.
See Lemma~\ref{lem:carets} for the meaning of {\em caret expansion}.

\begin{theorem}\label{them:carets-more} 
Let $f = \bigvee_{i} y_{i}x_{i}^{-1}$ and $g = \bigvee_{j} v_{j}u_{j}^{-1}$ in the monoid $D_{n}$.
Then $f \, \equiv \, g$ if and only if there is a prefix code $Z$ such that $Z$ is a caret extension of both $X$ and $U$  such that
$f \left( \bigvee_{z \in Z} zz^{-1} \right) = g \left( \bigvee_{z \in Z} zz^{-1} \right)$.
\end{theorem}
\begin{proof} Suppose that $f \, \equiv \, g$.
Then by Lemma~\ref{lem:new-characterization},
there exists $h \leq_{e} f,g$.
By Lemma~\ref{lem:needed}, we have that $h \leq_{e} f$ if and only if $\mathbf{d}(h) \leq_{e} \mathbf{d}(f)$ and $\mathbf{d}(h) \leq_{e} \mathbf{d}(g)$.
We now use Proposition~\ref{prop:cummings}, the prefix code corresponding to  $\mathbf{d}(h)$ is a caret expansion
of the prefix codes associated with both  $\mathbf{d}(f)$ and $\mathbf{d}(g)$.
The result now follows since $h\mathbf{d}(f) = h \mathbf{d}(g)$.
\end{proof}

We shall paraphrase the above theorem but need some terminology first.
Let $f = \bigvee_{i} y_{i}x_{i}^{-1}$.
Relabelling if necessary, we can write
$f =  y_{1}x_{1}^{-1} \vee \left( \bigvee_{i \neq 1} y_{i}x_{i}^{-1} \right)$.
Now consider the element
$g =  \left( \bigvee_{j=1}^{n}y_{1}a_{j}(x_{1}a_{j})^{-1}  \right) \vee \left( \bigvee_{i \neq 1} y_{i}x_{i}^{-1} \right)$.
We say that $g$ is obtained from $f$ by {\em inserting a caret} and that $f$ is obtained from $g$ by {\em deleting a caret}.
Clearly, $g \leq_{e} f$.
We may therefore rephrase Theorem~\ref{them:carets-more} using the above notions.

\begin{theorem}\label{them:carets-refine}
Let $f = \bigvee_{i} y_{i}x_{i}^{-1}$ and $g = \bigvee_{j} v_{j}u_{j}^{-1}$ in the monoid $D_{n}$.
Then $f \, \equiv \, g$ if and only if $g$ is obtained from $f$ by a finite sequence of inserting and deleting carets.
\end{theorem}

\subsection{Handling right-infinite strings}

The material in this section is fundamental to our main theorem proved in the next section.
The proof of the following is easy.

\begin{lemma}\label{lem:tea} Let $A_{n}$ be a finite alphabet where $n \geq 2$.
Then for finite strings $x$ and $y$ we have that
$xA_{n}^{\omega} = yA_{n}^{\omega}$ if and only if $x = y$.
\end{lemma}

It is convenient to handle elements of $A_{n}^{\omega}$ using special subsets of $A_{n}^{\ast}$:
\begin{itemize}
\item A subset $\mathscr{C} \subseteq A_{n}^{\ast}$ is called {\em good} if each pair of elements in $\mathscr{C}$
is prefix comparable and for every $m \in \mathbb{N}$ there exists $x \in \mathscr{C}$ such that $\mid x \mid = m$.
It is easy to check that there is a bijection between the elements of $A_{n}^{\omega}$ and the set of good subsets.
See, also, \cite{PP}.

\item A subset $X$ of $A_{n}^{\ast}$ is called a {\em filter} if it satisfies the following two conditions:
\begin{enumerate}
\item If $x,y \in A$ then there exist $u,v \in A_{n}^{\ast}$ such that $xu = yv \in A$; thus not only are $x$ and $y$ prefix comparable
but some right multiples belong to $A$.
\item If $x \in A$ and $x = yz$ then $y \in A$.
\end{enumerate}
A {\em maximal filter} is a filter that is not properly contained in another filter.

\item A subset $\{a_{1}, \ldots, a_{m} \} \subseteq aA_{n}^{\ast}$ is called a {\em tight cover} of $a$ if
whenever $z \in aA_{n}^{\ast}$ then there exists an $i$ such that $zr = a_{i}s$ for some $s,r \in A_{n}^{\ast}$.
A filter $A$ is said to be {\em tight} if $a \in A$ and $\{a_{1}, \ldots, a_{m} \}$ a tight cover of $a$ implies that $a_{i} \in A$
for some $i$.

\end{itemize}

\begin{example}\label{ex:finite-filter}
{\em Let $x$ be a string.
Denote by $\mathsf{Pref}(x)$ the set of all prefixes of $x$;
this set contains both $x$ and the empty string.
Then 
$\mathsf{Pref}(x)$
is a filter.
Clearly, it is finite and every finite filter is of this form.
}
\end{example}

Example~\ref{ex:finite-filter} above characterizes the finite filters.

\begin{proposition}\label{prop:tight-filters} In a free monoid, the following three classes of subsets are the same:
\begin{enumerate}
\item Good subsets.
\item Tight filters.
\item Maximal filters.
\end{enumerate}
\end{proposition}
\begin{proof} {\em Every good subset is a filter.}
Let $A$ be a good subset.
Let $x,y \in A$.
By definition, they are prefix comparable and so there exist strings $u$ and $v$ such that $z = xu = yv$.
Let $z' \in A$ be a string of length $\mid z \mid$
which exists since $A$ is a good subset.
Then since $z'$ is prefix comparable with all elements of $A$ and on the basis of length considerations
there are strings $u'$ and $v'$ such that $z' = xu' = yv'$.
This verifies the first condition in the definition of a filter.
Now, let $x \in A$ where $x = yz$.
Then, since $A$ is a good subset there exist strings $y'$ and $z'$ such that $x = y'z'$
where $\mid y' \mid = \mid y \mid$.
It follows that $y = y'$ and so all prefixes of $x$ belong to $A$
and we have verified the second condition for a filter.
We have therefore proved that every good subset is a filter.

{\em Every good subset is a tight filter.}
Let $A$ be a good subset.
Let $\{a_{1}, \ldots, a_{m} \} \subseteq aA_{n}^{\ast}$ be a tight cover of $a$ where $a \in A$.
We prove that $a_{i} \in A$ for some $i$.
Put $m$ equal to the maximum length of the elements $a_{1},\ldots, a_{m}$.
Since $A$ is a good subset there exists a unique $z \in A$ such that $\mid z \mid = m$.
But $a \in A$ and the length of $z$ is greater than or equal to the length of $a$
since all the $a_{i}$ have lengths no bigger than that of $a$.
It follows that $z = as$ for some $s \in A_{n}^{\ast}$.
But $\{a_{1}, \ldots, a_{m} \}$ is a tight cover of $a$ and so $zq = a_{i}r$ for some $i$ and $q,r \in A_{n}^{\ast}$.
By the choice of $z$, we know that $a_{i}$ is a prefix of $z$.
Using the fact that $A$ is a filter we deduce that $a_{i} \in A$, as required.
We have therefore proved that every good subset is a tight filter.

{\em Every tight filter is a good subset}.
Let $A \subseteq A_{n}^{\ast}$ be a tight filter.
By the definition of a filter, elements are pairwise prefix comparable.
We therefore have only to prove the following.
Let $m \in \mathbb{N}$ be arbitrary.
We shall construct an element $b \in A$ such that $\mid b \mid = m$.
Let $a \in A$ be arbitrary.
Put $n$ equal to the larger of $\mid a \mid$ and $m$.
We deal with the trivial case first.
Suppose that $\mid a \mid \geq m$.
Then $a$ has a prefix of size $m$ that belongs to $A$ and we are done.
In what follows, we therefore assume that $m > \mid a \mid$.
Let $p_{1}, \ldots, p_{s}$ be all the strings  of length $m \, - \mid a \mid$.
Then $ap_{1}, \ldots, ap_{s}$ are all the strings  of length $m$
beginning with $a$.
We shall prove in a moment that $\{ap_{1}, \ldots, ap_{s} \}$ is a tight cover of $a$.
Granted this, 
since $A$ is a tight filter, we must have that $ap_{i} \in A$ for some $i$ and we are done.
We now prove that 
$\{ap_{1}, \ldots, ap_{s} \}$ is a tight cover of $a$.
Let $z = ap$ for some $p \in A_{n}^{\ast}$.
There are two cases to consider.
Suppose first that $\mid p \mid \leq m \, - \mid a \mid$.
Then choose $u \in A_{n}^{\ast}$ such that $pu$ has length $m \, - \mid a \mid$.
Then $pu = p_{k}$ for some string $v$.
Thus $apu = ap_{k}$ and so $zu = ap_{k}$ and we are done.
Now, suppose that $\mid p \mid >  m\, - \mid a \mid$.
Then we can write $p = p_{k}u$ for some $k$ and some string $u$.
It follows that $ap = ap_{k}u$ and so $z = ap_{k}u$ and we are done.

{\em Every good subset is a maximal filter.}
We have proved that it is a filter so it is enough to prove that it is maximal.
Let $A$ be a good subset.
Suppose that $A \subseteq B$ where $B$ is a filter.
Let $x \in B$.
Then there exists $y \in A$ such that $\mid x \mid = \mid y \mid$.
But $x$ and $y$ are prefix comparable and so $x = y$.
Thus $B = A$.

{\em Every maximal filter is a good subset.}
Let $A$ be a maximal filter.
Since it is a filter its elements are pairwise prefix comparable.
Suppose that $A$ contain a string  $x$ of maximum length.
Then all elements of $A$ are prefixes of $x$;
that is, $A = \mathsf{Pref}(x)$.
Let $y$ be any string that has $x$ as a proper prefix.
Then $A = \mathsf{Pref}(x) \subseteq \mathsf{Pref}(y)$
is properly contained, which contradicts the fact that $A$ is a maximal fliter.
It follows that there is no bound on the lengths of elements of $A$ from which
we easily deduce that $A$ is a good subset.
\end{proof}

\subsection{The Boolean inverse monoid $C_{n}$}

By Lemma~\ref{lem:red-panda}, Proposition~\ref{prop:finverse} and Lemma~\ref{lem:skype},
the group of units of $D_{n}/\equiv$ is isomorphic to the Thompson group $G_{n,1}$.
It remains only to show that  $D_{n}/\equiv$ is Boolean.

\begin{proposition}\label{prop:tight-ultra}
In the inverse monoid $D_{n}$ every tight filter is an ultrafilter.
\end{proposition}
\begin{proof} We shall begin by connecting tight filters (respectively, maximal filters) in $\mathsf{E}(D_{n})$
with tight filters (respectively, maximal filters) in $A_{n}^{\ast}$.
Let $A$ be a filter in $A_{n}^{\ast}$.
Put $\mathsf{P}(A)$ equal to the set $\{xx^{-1} \colon x \in A\}^{\uparrow} \cap \mathsf{E}(D_{n})$
It is routine to check that $\mathsf{P}(A)$ is a prime filter.
Now, let $P$ be a prime filter in  $\mathsf{E}(D_{n})$.
Put $\mathsf{F}(P) = \{x \in A^{\ast}_{n} \colon xx^{-1} \in P\}$.
Since $P$ is a prime filter, the set $\mathsf{F}(P)$ is non-empty.
It is routine to check that it is a filter in $A_{n}^{\ast}$.
These two maps lead to an order-isomorphism between the set of filters in $A_{n}^{\ast}$
and the set of prime filters in $\mathsf{E}(D_{n})$.
We now prove that this correspondence leads to a bijection between the set of
maximal filters (respectively, tight filters) on $A_{n}^{\ast}$ and the set of maximal filters (respectively, tight filters) in $\mathsf{E}(D_{n})$.
It is routine to prove the former, and the latter is proved when we observe that
$\{a_{1}, \ldots, a_{m} \}$ is a tight cover of $a$ in $A_{n}^{\ast}$ precisely when 
$\{a_{1}a_{1}^{-1}, \ldots, a_{m}a_{m}^{-1}\}$ is a tight cover of $aa^{-1}$ in $D_{n}$.
By Proposition~\ref{prop:tight-filters}, the tight filters in $A_{n}^{\ast}$ are the same as the maximal filters.
It follows that the tight filters in $\mathsf{E}(D_{n})$ are the same as the maximal filters in $\mathsf{E}(D_{n})$.
The claim now follows by Proposition~\ref{prop:idpts-semigroup}.
\end{proof}

\noindent
{\bf Definition. }Put $C_{n} = D_{n}/\equiv$ which we call the {\em Cuntz inverse monoid}.\\

\begin{theorem}[The Cuntz inverse monoid]\label{them:one} The inverse monoid $C_{n}$
is a Boolean inverse $\wedge$-monoid whose group of units is the group
$G_{n,1}$.
\end{theorem}
\begin{proof}  By Proposition~\ref{prop:tight-ultra} tight filters and ultrafilters on $D_{n}$ are the same.
By Corollary~\ref{cor:idpt-pure}, the congruence $\equiv$ is idempotent-pure on $D_{n}$.
It now follows by Lemma~\ref{lem:dist-inv}, Lemma~\ref{lem:idpt-pure-dist} and Theorem~\ref{them:seven},
that $C_{n}$ is a Boolean inverse monoid.
The claim about the group of units follows by Proposition~\ref{prop:finverse} and Lemma~\ref{lem:red-panda}.
\end{proof}

We can now give a universal characterization of the Boolean inverse monoid $C_{n}$.
Denote by $\iota \colon P_{n} \rightarrow C_{n}$ the natural embedding.
The following result shows that the Cuntz inverse monoid $C_{n}$ is the {\em tight completion}
of the polycyclic inverse monoid $P_{n}$.
It is a special case of \cite{LV2020b}.
See also \cite{DM}.
A homomorphism $\theta \colon S \rightarrow B$ to a Boolean inverse semigroup is called s {\em cover-to-join}
map if whenever $\{a_{1},\ldots,a_{m} \}$ is a cover of $a$ we have that $\theta (a) = \bigvee_{i=1}^{m} \theta (a_{i})$.
The following theorem makes precise what we informally stated in Remark~\ref{rem:lagavulin}.

\begin{theorem}[The tight completion]\label{them:tight-completion} \mbox{}

\begin{enumerate}

\item Let $\theta \colon P_{n} \rightarrow B$ be a monoid homomorphism to a Boolean inverse monoid $B$ such that
$\bigvee_{i=1}^{n} \theta (a_{i}a_{i}^{-1}) = 1$.
Then it is a cover-to-join map.

\item There is a unique morphism of Boolean inverse monoids $\psi \colon C_{n} \rightarrow B$ such that
$\psi \iota = \theta$.

\end{enumerate}
\end{theorem}
\begin{proof} (1) It is easy to check that $\{a_{1}a_{1}^{-1},\ldots, a_{n}a_{n}^{-1} \}$ is a tight cover of $1$. 
By Lemma~\ref{lem:mpc} we have that
$\{x_{1}, \ldots, x_{m} \}$ is a maximal prefix code if and only if 
$\{x_{1}x_{1}^{-1}, \ldots, x_{m}x_{m}^{-1} \}$ is a tight cover of $\varepsilon \varepsilon^{-1}$.
Thus by Theorem~\ref{them:max-pc}, 
it follows that if
$\bigvee_{i=1}^{n} \theta (a_{i}a_{i}^{-1}) = 1$
then 
$\bigvee_{i=1}^{m} \theta (x_{i}x_{i}^{-1}) = 1$.
We now use Lemma~\ref{lem:corona},
we have that
if $\{x_{1}x_{1}^{-1}, \ldots, x_{m}x_{m}^{-1}\} \subseteq (xx^{-1})^{\downarrow}$, 
where $x_{i} = xp_{i}$,
then this is a tight cover if and only if 
 $\{p_{1}p_{1}^{-1}, \ldots, p_{m}p_{m}^{-1}\} \subseteq (\varepsilon \varepsilon^{-1})^{\downarrow}$ 
is a tight cover.
It follows that 
$\bigvee_{i=1}^{m} \theta (x_{i}x_{i}^{-1}) = \theta (xx^{-1})$.
By Lemma~\ref{lem:tight-covers} and the above result if
$\bigvee_{i} a_{i} \leq_{e} \bigvee_{j} b_{j}$ 
then
$\bigvee_{i} \theta (a_{i}) =  \bigvee_{j} \theta (b_{j})$.

(2) We now use Lemma~\ref{lem:new-characterization} to deduce that if 
$\bigvee_{i} a_{i} \equiv \bigvee_{j} b_{j}$ 
then 
$\bigvee_{i} \theta (a_{i}) =  \bigvee_{j} \theta (b_{j})$.
Denote the $\equiv$-class on $D_{n}$ containing the element $a$ by $[a]$.
Then we may define $\psi \colon C_{n} \rightarrow B$ by
$\psi ([\bigvee_{i} a_{i}]) = \bigvee_{i} \theta (a_{i})$.
This is well-defined by the above calculations.
It is now routine to check that $\psi$ has the requisite properties. 
\end{proof}

\begin{remark}{\em We can explain what is going on the above theorem as follows.
Let $\theta \colon P_{n} \rightarrow B$ be a cover-to-join map.
Then there is an essential morphism $\theta^{\ast} \colon \mathsf{D}(P_{n}) \rightarrow B$, by \cite[Proposition~6.14]{LL2},
such that $\theta^{\ast} \iota = \theta$.
But the distributive completion of $P_{n}$ is $D_{n}$ by Theorem~\ref{them:dist-completion}.
Thus we have a well behaved essential morphism $\theta^{\ast} \colon D_{n} \rightarrow B$.
But if $\theta^{\ast}(a) = \theta^{\ast}(b)$ then $a \, \equiv \, b$.
The result now follows.}
\end{remark}

We now compute the \'etale groupoid associated with $C_{n}$ under non-commutative Stone duality \cite{Lawson2010, Lawson2012, LL}
See \cite[Section 3]{Lawson2010} for further references.
Recall that the elements of the groupoid $\mathcal{G}_{n}$ associated with the Cuntz $C^{\ast}$-algebra are those elements of
$A_{n}^{\omega} \times \mathbb{Z} \times A_{n}^{\omega}$ of the form
$(xw, \mid x \mid - \mid y \mid, yw)$ where $x$ and $y$ are finite strings and $w \in A_{n}^{\omega}$.
The product is defined by $(z,k,z')(z',k', z'') = (z, k + k', z'')$ and $(z,k,z')^{-1} = (z',-k, z)$.
The topology on $\mathcal{G}_{n}$ has a basis of open sets of the form
$$U_{x,y,V} = \{(xw, \mid x \mid - \mid y \mid, yw) \colon w \in V\},$$ 
where $V$ is an open subset of the Cantor space $A_{n}^{\omega}$
and $x$ and $y$ are fixed finite strings.
The following theorem justifies the intuitive approach adopted in \cite[Section 3]{Lawson2010}.

\begin{theorem}[The associated groupoid]\label{them:associated-groupoid} 
The \'etale groupoid of $C_{n}$ is isomorphic to the groupoid $\mathcal{G}_{n}$.
\end{theorem}
\begin{proof}  By Theorem~\ref{them:isomorphism-etale-groupoids},
the \'etale topological groupoid associated with $C_{n}$ is isomorphic to the \'etale toplogical groupoid associated with $D_{n}$.
It follows that we can work with $D_{n}$ alone.
But, by Theorem~\ref{them:isomorphism-of-etale-groupoids}, we have that $\mathsf{G}(D_{n})$
is isomorphic to $\mathsf{G}(P_{n})$.
We shall accordingly prove that $\mathsf{G}(P_{n})$ is isomorphic as a topological groupoid with  $\mathcal{G}_{n}$.

Let $A$ be a proper filter in $P_{n}$.
We claim that $P_{n}$ is a linearly ordered set
To prove the claim, suppose that $xy^{-1}, uv^{-1} \in A$.
Then, since $A$ is a proper filter, there is a non-zero element $wz^{-1} \leq xy^{-1}, uv^{-1}$.
It follows that there are finite strings $p$ and $q$ such that
$w = xp = uq$ and $z = yp = vq$.
The finite strings $x$ and $u$ are comparable.
Without loss of generality, we can assume that $x = us$ for some finite string $s$.
Thus $q = sp$ and $y = vs$.
It follows that $xy^{-1} \leq uv^{-1}$.
The above calculation delivers something more:
$$\mid w \mid - \mid z \mid =   \mid xp \mid - \mid yp \mid = \mid uq \mid - \mid vq \mid.$$
 It follows that 
 $\mid x \mid - \mid y \mid = \mid u \mid - \mid v \mid = n$, say.
We shall call $n$ the {\em index} of the proper filter. 

Now, let $A$ be an ultrafilter.
Suppose that $A$ contained a minimum element $ab^{-1}$.
Then $A$ would contain only a finite number of elements
and would therefore be the principal filter $(ab^{-1})^{\uparrow}$.
But if $p$ is any non-empty finite string, then $(ap)(bp)^{-1} < ab^{-1}$.
It follows that $(ab^{-1})^{\uparrow}$ cannot be an ultrafilter.
Thus ultrafilters can have no smallest elements.
 
 Let $A$ be a linearly ordered set with no smallest element and closed upwards in $P_{n}$.
 Let $m'$ be an arbitrary element of $\mathbb{N}$.
 Then we claim that there exists $ab^{-1} \in A$ such that $\mid a \mid, \mid b \mid > m'$.
 We now prove the claim.
 Let $uv^{-1} \in A$.
 If $\mid u \mid, \mid v \mid > m'$ then we are done.
 We know that $A$ has no smallest element.
 Thus, since the set is linearly ordered, there is $u_{1}v_{1}^{-1} < uv^{-1}$ where $u_{1}v_{1}^{-1} \in A$.
 It follows that $(u_{1},v_{1}) = (u,v)p$ for some non-empty string $p$.
 Observe that $\mid u_{1} \mid > \mid u \mid$ and  $\mid v_{1} \mid > \mid v \mid$.
 Continuing in this way, we can find our element $ab^{-1}$.
 
 Let $B$ be any linearly ordered set of elements of $P_{n}$ with no smallest element which is closed upwards.
 We prove that it is an ultrafilter.
 Suppose not.
 Then there is a proper filter $A$ such that $B \subset A$.
 Let $ab^{-1} \in A \setminus B$.
 Let $m'$ be the larger of $\mid a \mid$ and $\mid b \mid$.
 Then we can find an element $uv^{-1} \in A$ such that
 $\mid u \mid, \mid v \mid > m'$.
 Since $ab^{-1}, uv^{-1} \in A$ and $A$ is a filter it follows that $ab^{-1}$ and $uv^{-1}$ are comparable.
 Thus either $ab^{-1} \leq uv^{-1}$ or $uv^{-1} \leq ab^{-1}$.
 But the latter would imply that $ab^{-1} \in B$ which contradicts our choice.
 But the former is impossible on size grounds.
 It follows that $A \setminus B$ is the empty set and so $A$ is an ultrafilter.
 
 We have therefore characterized the ultrafilters in $P_{n}$:
 they are the linearly ordered subsets with no smallest element which are also closed upwards.
 
 We now prove that there is a bijection between ultrafilters in $P_{n}$
 and elements of the groupoid $\mathcal{G}_{n}$.
   Let $(xw, n, yw)$ be an element of $\mathcal{G}_{n}$.
  Define $A$ to be the set of all elements above $xy^{-1}$ together with
  all elements of the form $xp(yp)^{-1}$ where $p$ is a finite prefix
  of the right-infinite string $w$.
  It is immediate from our above characterization that $A$ is an ultrafilter in $P_{n}$.
  Now, let $A$ be any ultrafilter in $P_{n}$.
  Let $xy^{-1} \in A$.
  Let $W$ be the set of all finite strings $p$ such that $xp(yp)^{-1} \in A$.
  From our calculations above, it is easy to check that $W$ is a good subset.
  Thus it corresponds to a right-infinite string $w$.
  Let $n$ be the index of $A$.
  Then we have constructed the triple $(xw, n, yw)$ which belongs to $\mathcal{G}_{n}$. 
  These two constructions are mutually inverse and so we have constructed a 
  bijection between $\mathsf{G}(P_{n})$ and $\mathcal{G}_{n}$.
  It is routine to check that this is an isomorphism of groupoids.
  
 We now turn to the topologies.
Fix an element $xy^{-1}$ of $P_{n}$ and consider the set of all ultrafilters that contain that element.
This is one of the basic open subsets in the topology of the associated groupoid.
Under our bijection established above,
it follows that the set of all ultrafilters containing $xy^{-1}$ corresponds to the set
$\{(xw, \mid x \mid - \mid y \mid, yw) \colon w \in A_{n}^{\omega} \}$
which is the set $U_{x,y,A_{n}^{\omega}}$.
Let $V \subseteq A_{n}^{\omega}$ be an arbitrary open set.
Then $V = XA_{n}^{\omega}$ where $X \subseteq A_{n}^{\ast}$ \cite{BP}.
Thus $V$ is a union of sets of the form $uA_{n}^{\omega}$ where $u$ is a finite string.
Now observe that the following set
$$\{(xuw, \mid xu \mid - \mid yu\mid, yuw) \colon w \in A_{n}^{\omega})$$ 
is equal to the set
$$\{(xw, \mid x \mid - \mid y \mid, yw) \colon w \in uA_{n}^{\omega}).$$ 
It follows that the topology with basis  $U_{x,y,A_{n}^{\omega}}$
is the same as the topology with basis the sets $U_{x,y,V}$. 
\end{proof}

\section{The $n$-ary Cantor algebra}

In this section, we shall describe a different approach to constructing the Thompson groups but
relate it to the one adopted in this paper.
This will involve combining the universal algebras introduced in \cite{Higman} 
with the apparently unrelated work in theoretical computer science in \cite{Statman}.\\

\noindent
{\bf Definition. }An {\em $n$-ary Cantor algebra} is a structure $(X,\alpha_{1}, \ldots, \alpha_{n}, \lambda)$, where
$\alpha_{1}, \ldots, \alpha_{n}$ are unary operations and $\lambda$ is an $n$-ary operation, satisfying the following two laws:
\begin{description}
\item[CA1] $(x\alpha_{1}, \ldots, x\alpha_{n})\lambda = x$ for all $x \in X$.
\item[CA2] $(x_{1},\ldots, x_{n})\lambda \alpha_{i} = x_{i}$ where $1 \leq i \leq n$.
\end{description}
Observe that we write the operators of the algebra to the right of their arguments.
Given such an algebra, we may define a bijection $\beta \colon X \rightarrow X^{n}$ by $x\beta = (x\alpha_{1}, \ldots, x\alpha_{n})$
for each $x \in X$. Conversely, every bijection from $X$ to $X^{n}$ defines an $n$-ary Cantor algebra \cite{Higman, Rosenthal}.
A {\em $2$-ary Cantor algebra}, or {\em binary Cantor algebra}, is referred to simply as a {\em Cantor algebra.}\\

We shall need a few simple results about $n$-ary Cantor algebras.
Let $Y$ be a subset of an $n$-ary Cantor algebra.
Consider the subalgebra generated by $Y$.
The elements of this subalgebra will be obtained from the subset $Y$ by applying a finite sequence of the algebra operations.
We define {\em allowable $\lambda$-expressions} over $Y$ inductively: the elements of $Y$ are allowable
$\lambda$-expressions; if $X_{1}, \ldots , X_{n}$ are allowable $\lambda$-expressions so too is 
$(X_{1}, \ldots, X_{n})\lambda$; finally, every allowable $\lambda$-expression is obtained by applying the previous two steps a finite number of times. 
Thus in the case of the binary Cantor algebra $(y_{1}, (y_{2}, y_{3})\lambda)\lambda$ is an allowable $\lambda$-expression but $(y_{1})\lambda$ is not.
The proof of the following follows by repeated application of law (CA2).

\begin{lemma}\label{lem:stonewall} 
Each element of the subalgebra of an $n$-ary Cantor algebra generated by $Y$ is equal to an allowable $\lambda$-expression.
\end{lemma}

The above is a sort of weak normal form for elements of the subalgebra generated by $Y$.
Lemma~\ref{lem:bush} below is just a version in our setting of a well-known result.

\begin{lemma}\label{lem:bush} 
Any allowable $\lambda$-expression determines a maximal prefix code over the alphabet $A_{n}$,
and conversely.
\end{lemma}

\begin{example}{\em We work with binary Cantor algebras.
Here $A_{n} = \{a,b\}$.
The allowable $\lambda$-expression $(y_{1}, (y_{2}, y_{3})\lambda)\lambda$ corresponds 
to the maximal prefix code $\{a, ba, bb\}$.}
\end{example}

The following result is key.

\begin{proposition}\label{prop:lincoln} 
We work in an arbitrary $n$-ary Cantor algebra.
Let $t$ be any element of an $n$-ary Cantor subalgebra generated by $Y$.
Let $t_{1}$ and $t_{2}$ be two allowable $\lambda$-expressions over $Y$ for $t$.
Then $t_{1}$ can be transformed to $t_{2}$ using only law (CA1).
\end{proposition}
\begin{proof} We use law (CA1) to convert both $t_{1}$ and $t_{2}$ into allowable $\lambda$-expressions $t_{1}'$ and $t_{2}'$, respectively, where the underlying maximal prefix codes
are uniform and have the same height.
By Lemma~\ref{lem:trump}, any two uniform maximal prefix codes of the same height are equal.
We now use (CA2) to deduce that the allowable $\lambda$-expressions  $t_{1}'$ and $t_{2}'$ are, in fact, identical.
We can therefore transform $t_{1}$ into $t_{2}$ using only law (CA1).
\end{proof}

\noindent
{\bf Goal of this section: }we shall construct an $n$-ary Cantor algebra $\mathscr{T}_{n}$ which will prove to be the free $n$-ary Cantor algebra on one generator  
and prove that the Thompson group $G_{n,1}$ is the group
of automorphisms of this algebra.
We shall construct the algebra $\mathscr{T}_{n}$ in two ways: 
first, following Statman, as a monoid and second, following Higman, as a set of labelled trees.

\subsection{A restriction monoid}

By a {\em restriction semigroup} $(S,E)$ or, simply, $S$, 
we mean a semigroup $S$ furnished with a commutative subsemigroup $E$ of the set of all idempotents $\mathsf{E}(S)$,
called the set of {\em projections},
equipped with a unary operation $a \mapsto a^{\ast}$ to the set $E$ satisfying the following axioms:
\begin{description}
\item[RS1] $e^{\ast} = e$ for each $e \in E$.
\item[RS2] $aa^{\ast} = a$.
\item[RS3] $(a^{\ast}b)^{\ast} = (ab)^{\ast}$.
\item[RS4] $a^{\ast}b = b(ab)^{\ast}$.
\end{description}

In Example~\ref{ex:ex1} and Example~\ref{ex:ex2}, 
we construct two examples of such restriction semigroups.

\begin{example}\label{ex:ex1}
{\em In an inverse semigroup, the operation $a \mapsto a^{-1}a$ has all the above properties
with $E$ being all idempotents.
This motivates our whole approach.
A special case is relevant to our work.
Let $X$ be a Boolean space.
Denote by $\mathcal{I}^{\scriptstyle cl}(X)$ the monoid of all homeomorphisms between the clopen subsets of $X$.
This is, in fact, a Boolean inverse monoid. The set of identity functions on the clopen subsets is the set of projections.}
\end{example}

\begin{example}\label{ex:ex2}
{\em We now construct an example of a restriction monoid that is not an inverse monoid.
Again, it is relevant to our work.
Let $X$ be a Boolean space.
Denote by $\mathcal{R}^{\scriptstyle cl}(X)$ the set of surjective local homeomorphisms between the clopen subsets of $X$. 
These functions are open, as is well-known, but they are also closed since their domains are compact and their images are Hausdorff \cite{Willard}.
It is therefore easy to see that $\mathcal{R}^{\scriptstyle cl}(X)$ is a monoid.
Choose for the set of projections the identity functions on the clopen subsets of $X$.
Call this set $E$.
In addition, each element of $\mathcal{R}^{\scriptstyle cl}(X)$ is a finite union of partial homeomorphisms.
If $f \in \mathcal{R}^{\scriptstyle cl}(X)$, define $f^{\ast}$ to be the identity function on the domain of definition of $f$.
It is clear that axioms (RS1), (RS2), (RS3) and (RS4) all hold.
}
\end{example}

There are a number of simple consequences of the above axioms:
\begin{enumerate}
\item For any projection $e$,  we have that $ae = a$ if and only if $a^{\ast} \leq e$.
\item Define $a \leq b$ if and only if $a = ba^{\ast}$.
Observe that $a \leq b$ if and only if $a = be$ for some idempotent $e$ using (1) above.
Then $\leq$ is a partial order on $S$, called the {\em natural partial order}, 
and because of axiom (RS4) it is compatible with the multiplication in $S$ on the right as well as on the left.
\item Suppose that $a,b \leq c$.
Then $a = ca^{\ast}$ and $b = cb^{\ast}$.
It follows that $ab^{\ast} = ba^{\ast}$.
\item Define $a \sim_{l} b$ if and only if $ab^{\ast} = ba^{\ast}$.
We call this the {\em left compatibility relation}.
Observe that $ab^{\ast} = ba^{\ast} = a \wedge b$.
\item If $a, b \in S$ are such that $a^{\ast}b^{\ast} = 0$ we say that $a$ and $b$ are {\em left orthogonal}.
Observe that if $a$ and $b$ are left orthogonal then $a \sim_{l} b$.
\item The projections form an order ideal with respect to the natural partial order.
\end{enumerate}

We shall work with a class of restriction semigroups $(S,E)$ that satisfy the following conditions:
\begin{description}
\item[H1] $S$ has an inverse subsemigroup $T$ (with the elements of $E$ being precisely the joins of finite subsets of $\mathsf{E}(T)$).
\item[H2] Each element of $S$ is a finite non-empty join of left compatible elements of $T$; in fact, left orthogonal elements.
\item[H3] If $a \leq \bigvee_{i} b_{i}$ where $a,b_{i} \in T$ then $a \leq b_{i}$ for some $i$.
\item[H4] $T$ is $E^{\ast}$-unitary.
\end{description}
We call such a restriction semigroup a {\em Higman semigroup}.

The generalizations of the relations $\equiv$ and $\leq_{e}$ to restriction semigroups pose no problems.
Let $S$ be a restriction emigroup.
Define the relation $\equiv$ on $S$ by $a \equiv b$ if and only if
$0 < x \leq a$ implies that $x^{\downarrow} \cap b^{\downarrow} \neq 0$ 
and
$0 < y \leq b$ implies that $y^{\downarrow} \cap a^{\downarrow} \neq 0$.
We write $a \leq_{e} b$ if for all $0 < x \leq b$ we have that $a \wedge x \neq 0$.
If this holds we say that $a$ is {\em essential} in $b$.
Observe that if $a \leq_{e} b$ then $a \equiv b$.

\begin{lemma}\label{lem:essential} 
In any restriction semigroup, we have that $a \leq_{e} b$ if and only if
$a^{\ast} \leq_{e} b^{\ast}$.
\end{lemma}
\begin{proof} Let $a \leq_{e} b$.
We prove that $a^{\ast} \leq_{e} b^{\ast}$.
Let $0 < e \leq b^{\ast}$.
Then $x = be \leq b$
and, since $b^{\ast} = e$, we know that $x$ is non-zero.
Both $x$ and $a$ are bounded above by $b$.
Thus the meet $x \wedge a$ exists and is non-zero.
But $a \sim_{l} x$ and so $(x \wedge a)^{\ast} = a^{\ast}e$ which is non-zero by assumption.
It follows that $x \wedge a$ is non-zero.
To prove the converse, suppose that $a^{\ast} \leq_{e} b^{\ast}$.
Let $0 < x \leq b$.
Then $0 < x^{\ast} \leq b^{\ast}$.
It follows that $a^{\ast}x^{\ast} \neq 0$.
But $a \sim_{l} x$ and so $(a \wedge x)^{\ast} = a^{\ast}x^{\ast}$.
Thus $a \wedge x$ is also non-zero.
\end{proof}

\begin{lemma}\label{lem:biscuits} Let $S$ be a Higman semigroup.
\begin{enumerate} 
\item $\equiv$ is a $0$-restricted congruence.
\item If $a \equiv b$ and $a$ is a projection then $b$ is a projection
\item If $a \equiv b$ then $a \sim_{l} b$.
\item $a \equiv b$ if and only if there exists $c \leq_{e} a,b$.
\end{enumerate}
\end{lemma}
\begin{proof} (1) It is clear that $\equiv$ is an equivalence relation.
Let $a \equiv b$ and let $c$ be any element.
Then both  $ac \equiv bc$ and $ca \equiv cb$ hold:
to prove that,  we devolve down into the inverse semigroup $S$ and use the properties there.
Let $a = \bigvee_{i} a_{i}$, $b = \bigvee_{j} b_{j}$ and $c = \bigvee_{k} c_{k}$
where $a_{i}, b_{j}, c_{k} \in T$.
Suppose that $a \equiv b$.
We prove that $\bigvee_{i,k} a_{i} c_{k} \equiv \bigvee_{j,k} b_{j}c_{k}$.
It is enough to work with elements of $T$.
Suppose that $0 < x \leq \bigvee_{i,k} a_{i} c_{k} $ where $x \in T$.
Then $x \leq a_{i}c_{k}$ for some $i$ and $k$.
We therefore have that $0 < xc_{k}^{-1} \leq a_{i}$.
By assumption, there exists some non-zero $z \in T$ and some $j$ such that $z \leq xc_{k}^{-1}, b_{j}$.
It follows that $zc_{k}$ is non-zero and $zc_{k} \leq x, b_{j}c_{k}$.
The proof that $ac \equiv bc$ now follows by symmetry as does the proof that $ca \equiv cb$.
The congruence $\equiv$ is $0$-restricted by construction.

(2) Suppose that $a \equiv b$ and $a$ is a projection.
Then $\bigvee_{i} e_{i} \equiv \bigvee_{j} b_{j}$ where $e_{i}, b_{j} \in T$
and the $e_{i}$ are projections.
For each $j$, we have that $b_{j}$ is a non-zero element below the righthand side.
It follows that there is an idempotent $e$ in $T$ such that $e \leq e_{i}, b_{j}$.
We now use the fact that $T$ is $E^{\ast}$-unitary to deduce that $b_{j}$ is a projection.
It follows that the righthand side is a projection.

(3) Suppose that $a \equiv b$. We shall prove that $a \sim_{l} b$.
Let $a = \bigvee_{i} a_{i}$ and $b = \bigvee_{j} b_{j}$ where $a_{i}, b_{j} \in T$.
Multiply both sides by $\mathbf{d}(a_{i})$ and use the fact that our joins are left orthogonal joins to get
$a_{i} \equiv \bigvee_{j} b_{j}\mathbf{d}(a_{i})$.
Now multiply each side by $a_{i}^{-1}$.
Using (2) above, we deduce that $b_{j}a_{i}^{-1}$ is a projection and so an idempotent.
Thus for each $i$ we always have that $b_{j}a_{i}^{-1}$ is an idempotent.
This proves that $a_{i} \sim_{l} b_{j}$; see, \cite[Section~1.4]{Lawson1998}.
The proof of the claim is now immediate.

(4) Suppose that $a \equiv b$.
Then $a \sim_{l} b$ by (3).
Thus $c = ab^{\ast} = ba^{\ast} = a \wedge b$.
We prove that $c \leq_{e} a$.
Let $0 < x \leq a$. 
Observe that $ab^{\ast} \equiv a$ and so $ab^{\ast}x^{\ast}  \equiv x$.
But $x \leq 0$ and $\equiv$ is $0$-restricted and so $ab^{\ast}x^{\ast} = (a \wedge b) \wedge x \neq 0$.
By symmetry, $c \leq_{e} b$.
The converse is immediate.
\end{proof}

Our next result attests to the uniqueness of the congruence $\equiv$ on Higman semigroups.

\begin{proposition}\label{prop:annie}
Let $\rho$ be any congruence on a Higman semigroup which has the following properties:
\begin{enumerate}
\item It is $0$-restricted.
\item If $a \, \rho \, b$ then $a \sim_{l} b$.
\item If $a \leq_{e} b$ then $a \, \rho \, b$.
\end{enumerate}
Then $\rho \,=\, \equiv$.
\end{proposition} 
\begin{proof} We prove first that $\equiv$ is contained in $\rho$.
Let $a \equiv b$.
Then  by Lemma~\ref{lem:biscuits}, there exists $c \leq_{e} a,b$.
Thus, by assumption, $c \, \rho \, a$ and $c \, \rho \, b$.
It follows that $a \, \rho\,  b$.
We now prove the converse.
Let $a \, \rho\,  b$.
By assumption, $a \sim_{l} b$.
We prove that 
$c = ab^{\ast} = ab^{\ast} \leq_{e} a,b$.
It is enough to prove that $c \leq_{e} a$.
Let $0 < x \leq a$.
We have that $x \sim_{l} c$.
Thus $x \wedge c = cx^{\ast}$.
It follows that $x \, \rho\, x \wedge c$.
We use the fact that $\rho$ is $0$-restricted to deduce that $x \wedge c \neq 0$.
It follows that $c \equiv a$ and $c \equiv b$ and so $a \equiv b$, as required.
\end{proof}

The following is now immediate by Lemma~\ref{lem:essential} and Lemma~\ref{lem:biscuits}.

\begin{lemma}\label{lem:sherlock} 
In a Higman semigroup, $a \equiv b$ implies that $a^{\ast} \equiv b^{\ast}$.
\end{lemma}

We shall work with a monoid $M_{n}$ constructed from the free monoid $A_{n}^{\ast}$ that contains the inverse monoid $D_{n}$
but which is itself not inverse.
Define $M_{n}$ to be the set of all surjective morphisms between finitely generated right ideals together with the empty function.
Suppose $\theta \colon R_{1} \rightarrow R_{2}$.
We can write $R_{1} = XA_{n}^{\ast}$ where $X$ is a finite prefix code in $A_{n}^{\ast}$. by Lemma~\cite[Lemma~A.1]{Birget}.
Put $Y = \theta (X)$.
Then 
 $\theta \colon XA_{n}^{\ast} \rightarrow YA_{n}^{\ast}$
 where $X$ is a prefix code and $Y = \theta (X)$.
 Observe that we cannot assume that $Y$ is a prefix code and so it is just an
 arbitrary finite subset of $A_{n}^{\ast}$.
 
 \begin{proposition}\label{prop:monoid} 
 $M_{n}$ is a Higman monoid 
 \end{proposition}
 \begin{proof} We use again the fact that the intersection of two finitely generated right ideals is a finitely generated right ideal.
 Let $\alpha \colon R_{1} \rightarrow R_{2}$ be a surjective morphism of finitely generated right ideals.
 Let $X'A_{n}^{\ast} \subseteq R_{1}$ be a finitely generated right ideal.
 Then $\alpha (X'A_{n}^{\ast}) = \alpha (X')A_{n}^{\ast}$ is a finitely generated right ideal.
 Now let $ZA_{n}^{\ast} \subseteq R_{2}$ be a finitely generated right ideal.
 Then $\alpha^{-1}(ZA_{n}^{\ast}) \subseteq R_{1}$ is a right ideal.
 We need to prove that it is finitely generated.
 Let $R_{2} = YA_{n}^{\ast}$ where $Y$ is a finite set,
 $R_{1} = XA_{n}^{\ast}$ where $X$ is a finite set,
 and $Y = \alpha (X)$.
 We have that $ZA_{n}^{\ast} \subseteq YA_{n}^{\ast}$.
 Let $z \in Z$.
 It is enough to prove that $\alpha^{-1}(zA_{n}^{\ast})$ is a finitely generated right ideal.
 We write $Y$ as a disjoint union $Y' \cup Y'' \cup Y'''$.
 The set $Y'''$ consists of those $y \in Y$ such that $z$ and $y$ are prefix incomparable.
 The set $Y'$ consists of those $y \in Y$ such that $z = yu$ for some string $u$ possibly empty.
 The set $Y''$ consists of those $y$ such that $y = zu$ for some non-empty string $u$. 
 Let $Y' = \{y_{1},\ldots, y_{s}\}$ where $z = y_{i}u_{i}$ and $1 \leq i \leq s$.
 Define $X_{i}$ to be the set $X_{i} = \alpha^{-1}(y_{i}) \cap X$ where $1 \leq i \leq s$.
 Let $Y'' = \{y_{s+1}, \ldots,  y_{s+j}, \ldots, y_{t}\}$ where $1 \leq j \leq t - s$ and $z$ is a proper prefix of $y_{s+j}$.
 Define $X_{j} = \alpha^{-1}(y_{s+j}) \cap X$ where $1 \leq j \leq t-s$.
 It is now routine to check that
 $$\alpha^{-1}(zA_{n}^{\ast}) = [(X_{1}u_{1} \cup \ldots \cup X_{s}u_{s}) \cup (X_{s+1} \cup \ldots \cup X_{t})]A_{n}^{\ast}.$$ 
 It is now easy to show that $M_{n}$ really is a monoid.
 It contains a copy of $P_{n}$.
  Let $\theta \colon XA_{n}^{\ast} \rightarrow YA_{n}^{\ast}$ be a surjective morphism.
 Let $X = \{x_{1}, \dots, x_{m}\}$ and $Y = \{y_{1}, \dots, y_{m}\}$ where we count multiplicities in $Y$.
 Let $x_{i}u \in XA_{n}^{\ast}$.
 Then $\theta (x_{i}u) = \theta (x_{i})u = y_{i}u$.
 It follows that $y_{i}x_{i}^{-1}(x_{i}u) = y_{i}u$.
 Thus $\theta = \bigcup_{i=1}^{m} y_{i}x_{i}^{-1}$.
 Suppose that $uv^{-1} \subseteq \bigcup_{i=1}^{m} y_{i}x_{i}^{-1}$.
Then $v = x_{i}p$ for some $x_{i} \in X$ and some string $p$.
We deduce that $u = y_{i}p$.
It follows that $uv^{-1} \leq y_{i}x_{i}^{-1}$ for some $i$.

 The intersection of any finite set of finitely generated right ideals is a finitely generated right ideals.
 The set of identity functions defined on such right ideals is therefore an appropriate set of projections.
 If $\theta \in M_{n}$, define $\theta^{\ast}$ to be the identity function defined on the domain of $\theta$.
 It remains to check that axioms (RS1)--(RS4) all hold.
 If the element of $M_{n}$ is $\bigvee_{i} y_{i}x_{i}^{-1}$ then
 $$\left( \bigvee_{i} y_{i}x_{i}^{-1} \right)^{\ast} = \bigvee_{i} x_{i}x_{i}^{-1}.$$
 Axioms (RS1) and (RS2) are easy to prove.
 We show that (RS3) holds.
 Let $a = \bigvee_{i} y_{i}x_{i}^{-1}$ and $b  = \bigvee_{j} v_{j}u_{j}^{-1}$. 
 Now $ab = \bigvee_{i,j} y_{i}x_{i}^{-1}v_{j}u_{j}^{-1}$ and $a^{\ast}b = \bigvee_{i,j} x_{i}x_{i}^{-1}v_{j}u_{j}^{-1}$.
 We now look at cases.
 Suppose that $v_{j} = x_{i}p$.
 Then 
 $y_{i}x_{i}^{-1}v_{j}u_{j}^{-1} = y_{i}pu_{j}^{-1}$
 and
 $x_{i}x_{i}^{-1}v_{j}u_{j}^{-1} = x_{i}pu_{j}^{-1}$.
 On symmetry grounds it is now clear that $(a^{\ast}b)^{\ast} = (ab)^{\ast}$.
 We show that (RS4) holds.
 Let $a = \bigvee_{i} y_{i}x_{i}^{-1}$ and $b  = \bigvee_{j} v_{j}u_{j}^{-1}$. 
 We have that
 $a^{\ast}b = \bigvee_{i,j} x_{i}x_{i}^{-1} v_{j}u_{j}^{-1}$
 and
 $b(ab)^{\ast} = \left( \bigvee_{j} v_{j}u_{j}^{-1} \right) \left( \bigvee_{i,k} x_{i} x_{i}^{-1} v_{k} u_{k}^{-1}  \right)$.
 It is clear that $a^{\ast}b \leq b(ab)^{\ast}$.
 The reverse inequality follows by looking at cases.
 \end{proof}
 
Each non-zero element of $M_{n}$ can be regarded as having a domain generated by a prefix code.
Thus it can be regarded as a tree whose leaves are labelled by elements of the free monoid:
the tree itself is nothing other than the usual tree representation of the prefix code $X$ and the leaves are the elements of $Y$.
We may therefore regard an element of $M_{n}$ as what we call a {\em symbol}
$\left( \frac{X}{Y} \right)$   
where $X = \{x_{1}, \ldots, x_{m}\}$ is a prefix code and $Y = \{y_{1}, \ldots, y_{m}\}$ is a subset of $A_{n}^{\ast}$ counting multiplicities.
Despite the notation, both sets $X$ and $Y$ are, in fact, ordered.
The actual element of $M_{n}$ is then the partial function
$$F \left( \frac{X}{Y} \right)   = \bigcup_{i=1}^{m} y_{i}x_{i}^{-1}$$
where $F \left( \frac{X}{Y} \right) \colon XA_{n}^{\ast} \rightarrow YA_{n}^{\ast}$ is given by
$F \left( \frac{X}{Y} \right) (x_{i}u) = y_{i}u$ where $u \in A_{n}^{\ast}$.
It is useful to define
$\left( \frac{\varnothing}{Y} \right)$
to be the empty partial function.   
Lemma~\ref{lem:needed-nec} below shows that we can regard elements of the monoid $M_{n}$ as symbols.
This will become important later.

\begin{lemma}\label{lem:needed-nec} 
There is a bijection between the set of symbols and the set of elements of $M_{n}$.
\end{lemma}
\begin{proof} Suppose that $\bigvee_{i=1}^{p} y_{i}x_{i}^{-1} = \bigvee_{j=1}^{q} v_{j}u_{j}^{-1}$.
We shall prove directly that their respective symbols are equal.
For each $i$, there exists $j$ such that $x_{i}x_{i}^{-1} \leq u_{j}u_{j}^{-1}$.
Thus $x_{i}$ is a prefix of $u_{j}$.
By the same token, $u_{j}$ is a prefix of some $x_{k}$.
It now follows that $X = U$.
Once this is established it is routine to check that $Y = V$.
\end{proof}

We shall now define a homomorphism from the monoid $M_{n}$ to the monoid
of all surjective local homeomorphisms between the clopen subsets of $A_{n}^{\omega}$.
As above,  we denote by  $F\left( \frac{X}{Y} \right)$ the partial function induced on the set $A_{n}^{\omega}$ by
 the symbol $\left( \frac{X}{Y} \right)$.
 The proof of the following is just a special case of \cite[Lemma~9.16]{LV2020a}.

\begin{lemma}\label{lem:watson} In the monoid $M_{n}$ we have that
$\bigvee_{i}x_{i}x_{i}^{-1} \equiv \bigvee_{j} y_{j}y_{j}^{-1}$ if and only if $XA_{n}^{\omega} = YA_{n}^{\omega}$.
\end{lemma}

Proposition~\ref{prop:coffee} below connects symbols to elements of $M_{n}$

\begin{proposition}\label{prop:coffee} In the monoid $M_{n}$, we have that 
$\left( \frac{X}{Y} \right)
\equiv
\left( \frac{U}{V} \right)$
if and only if
$F \left( \frac{X}{Y} \right)
=
F \left( \frac{U}{V} \right)$.
\end{proposition}
\begin{proof}
Suppose first that 
$\left( \frac{X}{Y} \right)
\equiv
\left( \frac{U}{V} \right)$.
Then by Lemma~\ref{lem:sherlock}, we have that
$\left( \frac{X}{Y} \right)^{\ast}
\equiv
\left( \frac{U}{V} \right)^{\ast}$.
Thus, by Lemma~\ref{lem:watson}, we have that $XA_{n}^{\omega} = UA_{n}^{\omega}$.
It follows that the domains of definition of 
$F \left( \frac{X}{Y} \right)$
and
$F \left( \frac{U}{V} \right)$
are the same.
Let $x_{i}w \in XA_{n}^{\omega}$.
Then $F\left( \frac{X}{Y} \right) (x_{i}w) = y_{i}w$.
Choose a prefix $x$ of $w$ such that $x_{i}x$ has length greater than any element of $U$
and such that $y_{i}x$ has a greater length than any element of $V$.
We have that $y_{i}x(x_{i}x)^{-1} \leq y_{i}x_{i}^{-1}$.
By assumption, there exists $ab^{-1} \leq y_{i}x(x_{i}x)^{-1}, v_{j}u_{j}^{-1}$ for some $j$.
Thus $a = y_{i}xp = v_{j}q$ and $b = x_{i}xp = u_{j}q$.
Because of our choice of lengths, we have that
$y_{i}x = v_{j}s$ and $x_{i}x = u_{j}t$.
We can easily see that $s = t$.
It is now routine to check that
$F\left( \frac{X}{Y} \right) (x_{i}w) = F\left( \frac{U}{V} \right) (x_{i}w)$.

We now prove the converse.
Let $ab^{-1} \leq y_{i}x_{i}^{-1}$ for some $i$.
Then $a = y_{i}p$ and $b = x_{i}p$.
Let $w \in A_{n}^{\omega}$.
Choose a prefix $x$ of $w$ such that $x_{i}px$ is longer than any element in $U$
and such that $y_{i}px$ is longer than any element in $V$.
By assumption, we quickly deduce that there exist finite strings $s$ and $t$ such that
$x_{i}px = u_{j}s$ and $y_{i}px = v_{j}t$.
Let $\hat{w}$ be any right-infinite string.
Using Lemma~\ref{lem:tea} and our assumption, we deduce that $s = t$.
We then get that $y_{i}px(x_{i}px)^{-1} \leq ab^{-1}, v_{j}u_{j}^{-1}$.
The result now follows by symmetry.
\end{proof}

\begin{proposition}\label{prop:light} Let $(S,E)$ be a restriction semigroup.
Let $\rho$ be a congruence on $S$ such that $a \, \rho \, b$ implies that $a^{\ast} \, \rho \, b^{\ast}$, 
and $e \, \rho \, a$, where $e$ is a projection, implies that $a$ is a projection.
Then $(S/ \rho, E/\rho)$ is a restriction semigroup.
\end{proposition}
\begin{proof} We have that $S/\rho$ is a semigroup and $E/\rho$ is a commutative idempotent subsemigroup.
Define $\rho (a)^{\ast} = \rho (a^{\ast})$.
\end{proof}

We now apply Proposition~\ref{prop:monoid}, Lemma~\ref{lem:biscuits}, Lemma~\ref{lem:sherlock}, Proposition~\ref{prop:light}
to deduce that $M_{n}/\equiv$ is also a restriction semigroup.\\

\noindent
{\bf Definition.} Put $\mathscr{C}_{n} = M_{n}/\equiv$.\\

The monoid $\mathscr{C}_{n}$ will play an important r\^ole in what we do next.

We shall make our definitions first on the monoid $M_{n}$.
Let $\left( \frac{X_{1}}{Y_{1}} \right), \ldots, \left( \frac{X_{n}}{Y_{n}} \right)$ be $n$ symbols.
Each of $X_{1}, \ldots, X_{n}$ is a prefix code.
Thus by Lemma~\ref{lem:prefix-codes} the set $X = a_{1}X_{1} \cup \ldots \cup a_{n}X_{n}$ is a prefix code.
By $Y = Y_{1} \cup \ldots \cup Y_{n}$, we mean a union that preserves the order of the terms.
Define
$$\left( \left( \frac{X_{1}}{Y_{1}} \right), \ldots,  \left( \frac{X_{n}}{Y_{n}} \right) \right) \lambda = \left( \frac{X}{Y} \right).$$         
We have therefore defined an $n$-ary operation on the set $M_{n}$.
We now define $n$ unary operations on the set of symbols.
Choose $i$.
Let
$\left( \frac{X}{Y} \right)$
be a symbol.
There are two cases.
Suppose first that $X \neq \{ \varepsilon \}$.
Then by Lemma~\ref{lem:quotients-pc}, we have that 
$\left( \frac{a_{i}^{-1}X}{Y} \right)$
is a well-defined symbol;
in this case, define
$$\left( \frac{X}{Y} \right)\alpha_{i} = \left( \frac{a_{i}^{-1}X}{Y} \right).$$
We now deal with the remaining case.
Define
$$\left( \frac{\varepsilon}{y} \right)\alpha_{i} = \left( \frac{\varepsilon}{ya_{i}} \right).$$
The proofs of Lemma~\ref{lem:stacia} are routine.
They imply that the operations on $M_{n}$ can safely be defined on $\mathscr{C}_{n}$.

\begin{lemma}\label{lem:stacia} In the monoid $M_{n}$, we have the following.
\begin{enumerate}
\item If $\left( \frac{X}{Y} \right) \equiv \left( \frac{X'}{Y'} \right)$ then
$\left( \frac{aX}{Y} \right) \equiv \left( \frac{aX'}{Y'} \right)$
for each $a \in A_{n}$.
\item If $\left( \frac{X}{Y} \right) \equiv \left( \frac{X'}{Y'} \right)$ then
$\left( \frac{X}{Y} \right) \alpha_{i} \equiv \left( \frac{X'}{Y'} \right) \alpha_{i}$
in all cases.
\end{enumerate}
\end{lemma}

\subsection{$n$-ary Cantor algebras}

We shall now prove that 
$$\left(\mathscr{C}_{n}, \alpha_{1}, \ldots, \alpha_{n}, \lambda \right)$$
 is an $n$-ary Cantor algebra.
We now need to regard $A_{n}^{\omega}$ as the Cantor space.
For each $a \in A_{n}$ define $\rho_{a}$ to be the partial homeomorphism of 
 $A_{n}^{\omega}$
given by $w \mapsto wa$.
Thus $\rho_{a} \colon A_{n}^{\omega} \rightarrow aA_{n}^{\omega}$ is a homeomorphism.
We now define an $n$-ary map $\lambda$ on $\mathcal{R}^{\scriptstyle cl}(X)$ 
and unary maps $\alpha_{1}, \ldots, \alpha_{n}$.
Define the $n$-ary map $\lambda$ by 
$(f_{1}, \ldots, f_{n}) \lambda = f_{1} \rho_{a_{1}}^{-1} \cup \ldots \cup f_{n} \rho_{a_{n}}^{-1}$
and define the unary map $(f) \alpha_{i} = f \rho_{a_{i}}$.
Observe that $\rho_{a}^{-1}\rho_{a}$ is the identity function on $A_{n}^{\omega}$ 
and that $\rho_{a_{1}} \rho_{a_{1}}^{-1} \cup \ldots \cup \rho_{a_{n}} \rho_{a_{a}}^{-1}$ is the identity function on $A_{n}^{\omega}$.
The proof of Lemma~\ref{lem:gregson} below is now straightforward.

\begin{lemma}\label{lem:gregson} With the above definitions,
$\mathcal{R}^{\scriptstyle cl}(A_{n}^{\omega})$ 
is an $n$-ary Cantor algebra.
\end{lemma}

It is not immediately obvious that the operations we have defined above agree with the operations we have defined on $\mathscr{C}_{n}$.
Lemma~\ref{lem:results-barnard}, whose proof is routine, shows that these are the same operations.

\begin{lemma}\label{lem:results-barnard}
Let $a \in A_{n}$ be a letter.
\begin{enumerate}
\item 
If $X \neq \varepsilon$ then $\left( \frac{X}{Y} \right) \left( \frac{\varepsilon}{a} \right) = \left( \frac{a^{-1}X}{Y} \right)$
whereas 
if $X = \varepsilon$ then $\left( \frac{\varepsilon }{y} \right) \left( \frac{\varepsilon}{a} \right)  = \left( \frac{\varepsilon}{ya} \right)$.
\item $\left( \frac{X}{Y} \right) \left( \frac{a}{\varepsilon} \right) = \left( \frac{aX}{Y} \right)$
\end{enumerate}
\end{lemma}

We now summarize what we have proved so far.

\begin{theorem}\label{them:higman-one}
We define the following operations in $\mathscr{C}_{n}$.
Define the $n$-ary operation $\lambda$ by
$$(f_{1}, \ldots, f_{n}) \lambda = f_{1} \rho_{a_{1}}^{-1} \cup \ldots \cup f_{n} \rho_{a_{n}}^{-1}$$
and the $n$ unary operations $\alpha_{1}, \ldots, \alpha_{n}$ by
 $$(f) \alpha_{i} = f \rho_{a_{i}}.$$
 Then
 $$\left(\mathscr{C}_{n}, \alpha_{1}, \ldots, \alpha_{n}, \lambda \right)$$
is an $n$-ary Cantor algebra.
 \end{theorem}

The fact that $\mathscr{C}_{n}$ can be regarded as an $n$-ary Cantor algebra is due to Statman \cite{Statman} and is the monoid he refers to as POPS.

In the light of Lemma~\ref{lem:essential}, Proposition~\ref{prop:annie}, and Proposition~\ref{prop:cummings},
we may reprove Theorem~\ref{them:carets-more} and Theorem~\ref{them:carets-refine} as Theorem~\ref{them:carets-refine-new}.

\begin{theorem}\label{them:carets-refine-new}
Let $f = \bigvee_{i} y_{i}x_{i}^{-1}$ and $g = \bigvee_{j} v_{j}u_{j}^{-1}$ in the monoid $M_{n}$.
Then $f \, \equiv \, g$ if and only if $g$ is obtained from $f$ by a finite sequence of inserting and deleting carets.
\end{theorem}

\subsection{The $1$-generated case}

We proved in the previous section that
$$\left(\mathscr{C}_{n}, \alpha_{1}, \ldots, \alpha_{n}, \lambda \right)$$
 is an $n$-ary Cantor algebra.
This monoid has a zero which is something we so not want.
So, in this section, we shall restrict our attention to a submonoid of $\mathscr{C}_{n}$: namely, the total maps.
This is still an $n$-ary Cantor algebra and we denote it by  
$$\left(\mathscr{T}_{n}, \alpha_{1}, \ldots, \alpha_{n}, \lambda \right).$$
The fact that it is a monoid can be easily checked from the definition of a restriction monoid
since $f \in \mathscr{T}_{n}$ precisely when $f^{\ast}$ is equal to the identity map on $A_{n}^{\omega}$.
This is the monoid that Statman calls TOPS \cite{Statman}.
Observe that a symbol $\left( \frac{X}{Y}\right)$ represents an element of $\mathscr{T}_{n}$ precisely when $X$ is now a {\em maximal} prefix code
by Proposition~\ref{prop:mpc}.
In particular, the monoid  $\mathscr{T}_{n}$ does not contain zero.
A string $x$ in $A_{n}^{\ast}$ will represent the element $x\varepsilon^{-1}$ of the monoid $\mathscr{T}_{n}$.
We prove first that every element in $\mathscr{T}_{n}$ can be written using the $n$-ary operation $\lambda$ and elements of the free monoid.

\begin{lemma}\label{lem:normal-forms} 
Every element in $\mathscr{T}_{n}$ can be written as an allowable $\lambda$-expression over the free monoid.
\end{lemma}
\begin{proof} Let $f$ be any non-identity element of $\mathscr{T}_{n}$.
Then $f$ can be represented by a symbol $\left( \frac{X}{Y} \right)$ where $X$ is a maximal prefix code.
The proof will be by induction on the number of elements of $X$.
If $X$ contains just one element it is the trivial maximal prefix code.
Thus $f = y\varepsilon^{-1}$.
We can therefore represent $f$ by $y$, an element of the free monoid $A_{n}^{\ast}$.
In what follows, we therefore assume that $X$ contains more than one element.
From the structure of maximal prefix codes Theorem~\ref{them:max-pc},
there exists $x \in X$ such that $xA_{n} \subseteq X$.
Put $X' = X \setminus xA_{n} \cup \{x\}$ which is a maximal prefix code of cardinality strictly smaller than that of $X$.
We may write $X = (X' \setminus \{x\}) \cup xA_{n}$.
Thus, relabelling if necessary, $f = (y_{1}, \ldots, y_{n})\lambda x^{-1} \cup f'$.
It follows that we have written $f = \bigvee_{x \in X'} f_{x}x^{-1}$ where $X'$ is a strictly smaller maximal prefix code and the $f_{x}$
are either elements of the free monoid or constructed using $\lambda$.
\end{proof}

The above lemma is illustrated by the following lemma.

\begin{example}{\em We work over the alphabet $A_{2} = \{a,b\}$.
Consider the element $f = a^{2}a^{-1} \vee a^{2} (ba)^{-1} \vee abb^{-2}$.
This is based on the maximal prefix code $\{a,ba,bb\}$.
This is equal to $a^{2}a^{-1} \vee (a^{2}, ab)\lambda b^{-1}$
which in turn is equal to $(a^{2}, (a^{2}, ab)\lambda)\lambda$.}
\end{example}

\begin{theorem}\label{them:free} The $n$-ary Cantor algebra $\left(\mathscr{T}_{n}, \alpha_{1}, \ldots, \alpha_{n}, \lambda \right)$
is the free $n$-ary Cantor algebra on one generator.
\end{theorem}
\begin{proof} We prove first that $\left(\mathscr{T}_{n}, \alpha_{1}, \ldots, \alpha_{n}, \lambda \right)$
is generated by one element; namely, the identity.
We shall use Lemma~\ref{lem:normal-forms}.
Observe that if $y$ is any string then it is obtained from $1$ by applying the correct unary operations.
Thus each element of $\mathscr{T}_{n}$ can be written as a term involving $\lambda$s and elements of the free monoid $A_{n}^{\ast}$.
But each element of the free monoid is simply equal to a sequence of $\alpha$s applied to the monoid identity $1$.

We now prove that it is the free such algebra.
Let $(C,\beta_{1}, \ldots, \beta_{n}, \gamma)$ be any non-empty $n$-ary Cantor algebra
generated by the element $x$.
Define the map that takes $1$ in $\mathscr{T}_{n}$ to the $x$ in $C$.
We map $a_{i}$ to the element $x\alpha_{i}$. 
We use Theorem~\ref{them:carets-refine}.
Since we are working in an $n$-ary Cantor algebra, we have the identity
$$(ua_{1}, \ldots, ua_{n})\lambda = u$$
for any string $u$.
Inserting and deleting such expressions corresponds to inserting and deleting carets.

Let $t$ be any element of  $(C,\beta_{1}, \ldots, \beta_{n}, \gamma)$.
By using the axioms, it can be written as an allowable $\lambda$-expression over $x$.
This allowable form is the image of some symbol.
On the other hand, $t$ could be written as an allowable $\lambda$-expression over $x$ in some other way.
By Proposition~\ref{prop:lincoln}, these two expressions can be converted one into the other only by the insertion or deletion of carets.
It follows that the two elements of $M_{n}$ are $\equiv$-related.
Thus there is a well-defined map from $\mathscr{T}_{n}$ to $C$.
This map is also onto by Lemma~\ref{lem:stonewall}.
From the way this map is defined, it is clear that it is a homomorphism of $n$-ary Cantor algebras.
Finally, it is clearly the unique such homomorphism mapping the identity to $x$.
\end{proof}

\begin{theorem}\label{them:higman-two}
The automorphism group of the $n$-ary Cantor algebra $\mathscr{T}_{n}$ is the Thompson group $G_{n_1}$.
\end{theorem}
\begin{proof} The group of units of $\mathscr{T}_{n}$ is the Thompson group $G_{n_1}$.
Left multiplication by an element of the group of units is an automorphism of the  $n$-ary Cantor algebra $\mathscr{T}_{n}$.
We now prove the converse.
Let $\theta$ be an automorphism of  the  $n$-ary Cantor algebra $\mathscr{T}_{n}$.
We shall prove that $\theta$ is simply left multiplication by an element of the group of units of $\mathscr{T}_{n}$,
the element in question being $\theta ( 1 )$.
We have that $1 \equiv a_{1}a_{1}^{-1} \vee \ldots \vee a_{n}a_{n}^{-1}$.
Thus
$1 \equiv  (a_{1}, \ldots, a_{n})\lambda$.
We now apply $\theta$ to both sides and use the fact that $\theta$ is an automorphism to get that
$$\theta (1) = \theta (a_{1})a_{1}^{-1} \vee \ldots \vee \theta (a_{n})a_{n}^{-1}.$$
Now let $f = \bigvee_{i=1}^{p} y_{i}x_{i}^{-1}$ be any element of $\mathscr{T}_{n}$.
We prove that
$\theta (f) =   \bigvee_{i=1}^{p} \theta (y_{i}) x_{i}^{-1}$.
The first step is to use the relation
$a_{1}a_{1}^{-1} \vee \ldots \vee a_{n}a_{n}^{-1} \equiv \varepsilon \varepsilon^{-1}$ 
to ensure that all the strings $x_{1}, \ldots, x_{p}$ have the same length.
We can write $f$ as
$f = f_{1}a_{1}^{-1} \vee \ldots \vee f_{n}a_{n}^{-1}$
since it is defined using a maximal prefix code.
We now use the fact that $\theta$ is an automorphism to deduce that
$\theta (f) = \theta (f_{1})a_{1}^{-1} \vee \ldots \vee \theta (f_{n})a_{n}^{-1}$.
We now repeat this process and the result now follows.
Let $y$ be any non-empty string.
Define $y = y'\bar{y}$ where $y' \in A_{n}$ is a letter.
Then, since $\theta$ is an automorphism we have that $\theta (y) = \theta (y')\theta (\bar{y})$.
It follows that $\theta$ is determined once its value on the $n$ letters is known.
We now prove that
$$\theta \left( \bigvee_{j=1}^{s} y_{j}x_{j}^{-1}   \right) = \left( \bigvee_{i=1}^{n} \theta (a_{i})a_{i}^{-1}  \right) \left(   \bigvee_{j=1}^{s} y_{j}x_{j}^{-1}  \right).$$
We have that 
$$\theta \left( \bigvee_{j=1}^{s} y_{j}x_{j}^{-1}   \right)  =  \bigvee_{j=1}^{s} \theta (y_{i}') \bar{y_{i}} x_{j}^{-1}.$$ 
We now calculate the product
$$\left( \bigvee_{i=1}^{n} \theta (a_{i})a_{i}^{-1}  \right) \left(   \bigvee_{j=1}^{s} y_{j}x_{j}^{-1}  \right).$$
We focus on the product $\theta (a_{i})a_{i}^{-1}y_{j}x_{j}^{-1}$.
If $y_{j}$ begins with the letter $a_{i}$ then the product is $\theta (y_{j}')\bar{y}_{j}x_{j}^{-1}$ otherwise it is zero.
We have therefore proved that $\theta$ is represented by left multiplication by the element
$$\theta (1) = \theta (a_{1})a_{1}^{-1} \vee \ldots \vee \theta (a_{n})a_{n}^{-1}.$$
It remains to prove that this element is invertible.
Put $\theta (1) = g$.
Then, by the above, $\theta (f) = gf$.
We prove that $g$ is an element of the group of units of $\mathscr{T}_{n}$.
The function $\theta$ is surjective and so there exists an element $h \in \mathscr{T}_{n}$ such that $\theta (h) = 1$.
Thus $gh = 1$.
Now $ghg = g$ and so $\theta (hg) = g$.
But $\theta (1) = g$.
It follows that $hg = 1$ since $\theta$ is injective.
We have therefore proved that $g$ is invertible in $\mathscr{T}_{n}$.
It is therefore an element of the Thompson group $G_{n,1}$.
\end{proof}

We conclude this paper by determining normal forms for the elements of $\mathscr{T}_{n}$ and thereby make the connection with the work of Higman \cite{Higman}.
An element of $\mathscr{T}_{n}$ is determined by a symbol $\frac{X}{Y}$ where $X$ is a maximal prefix code and $Y$ is any subset of $A_{n}^{\ast}$.
In general, of course, there will be infinitely many symbols that determine the same element of  $\mathscr{T}_{n}$.
We want to pick a single element form the $\equiv$-class of $\frac{X}{Y}$.
To do this, we use Theorem~\ref{them:carets-refine}.
Given a symbol $\left( \frac{X}{Y} \right)$, we say it can be {\em simplified}
if there are $ya_{1}, \ldots, ya_{n} \in Y$ and $xa_{1} \ldots, xa_{n} \in X$ such that 
$ya_{i}$ is matched with $xa_{i}$ for $1 \leq i \leq n$.
Remove the occurrences  $ya_{1}, \ldots, ya_{n}$ from $Y$ to obtain $Y'$ 
and the corresponding occurrences $xa_{1} \ldots, xa_{n}$ from $X$ to obtain $X'$.
We therefore obtain the symbol $\left( \frac{X' \cup \{x\}}{Y' \cup \{ y\}} \right)$
which is $\equiv$-related to $\left( \frac{X}{Y} \right)$.
By a {\em standard symbol} we mean one $\left( \frac{X}{Y} \right)$ where $X$ is a maximal prefix code
which cannot be simplified any further.
We claim that two standard symbols are $\equiv$-related if and only if they are equal;
this follows by Theorem~\ref{them:carets-refine} and \cite[Proposition 1.4]{Birgetarxiv}
where we observe that we must use the arXiv version of Birget's paper to find the proof.
We claim that the standard symbols defined above are in bijective correspondence
with the standard forms over $\{x\}$ defined in \cite{Higman}.
The only difference is that to obtain standard forms over $\{x\}$
we prefix each element of $yA_{n}^{\ast}$ by the letter $x$ which we assume not to belongs to $A_{n}^{\ast}$.

We now describe the form taken by the operations in the $n$-ary Cantor algebra when applied to the standard symbols.
The standard symbol $\left( \frac{\varepsilon}{y} \right)$ determines the globally defined mapped of $A_{n}^{\omega}$ given by $w \mapsto yw$.
We have that $\left( \frac{\varepsilon}{y} \right)\alpha_{i} = \left( \frac{\varepsilon}{ya_{i}} \right)$.
Let $f_{1}, \ldots, f_{n}$ be standard symbols.
We write $f_{i} = \left( \frac{X_{i}}{Y_{i}} \right)$.
Thus $(f_{1}, \ldots, f_{n})\lambda$ has the symbol
$\left(\frac{a_{1}X_{1} \cup \ldots \cup a_{n}X_{n}} {Y_{1} \cup \ldots \cup Y_{n}}\right)$.
If we now apply the unary operation $\alpha_{i}$ to this symbol we get
$\left(\frac{a_{i}^{-1}a_{1}X_{1} \cup \ldots \cup a_{i}^{-1}a_{n}X_{n}} {Y_{1} \cup \ldots \cup Y_{n}}\right)$.
This is just $f_{i} = \left( \frac{X_{i}}{Y_{i}} \right)$.
We now consider the operation $\lambda$.
Let $f_{1}, \ldots, f_{n}$ be standard symbols.
We write $f_{i} = \left( \frac{X_{i}}{Y_{i}} \right)$.
Thus $(f_{1}, \ldots, f_{n})\lambda$ has the symbol
$\left(\frac{a_{1}X_{1} \cup \ldots \cup a_{n}X_{n}} {Y_{1} \cup \ldots \cup Y_{n}}\right)$.
This might not be a standard symbol.
Suppose that $f_{i} = f\alpha_{i}$ where $f = \left( \frac{X}{Y} \right)$.
Thus $(f_{1}, \ldots, f_{n})\lambda$ has the symbol
$\left(\frac{a_{1}a_{1}^{-1}X  \cup \ldots \cup a_{n}a_{n}^{-1}X} {Y \cup \ldots \cup Y}\right)$.
But this is $\equiv$-equivalent to the standard symbol $f = \left( \frac{X}{Y} \right)$.
We see that the $\equiv$-relation is hidden in \cite{Higman} but we have made it explicit here.
We have proved the following theorem.

\begin{theorem}\label{them:higman-forms}
Let $A_{n}$ be a finite alphabet with $n \geq 2$.
Then there is an isomorphism between the $n$-ary Cantor algebra
$\left(\mathscr{T}_{n}, \alpha_{1}, \ldots, \alpha_{n}, \lambda \right)$
and Higman's $n$-ary Cantor algebra of standard forms over the set $\{x\}$.
In particular, the standard symbols of $\mathscr{T}_{n}$ are in bijective correspondence with the standard forms over the set $\{x\}$.
\end{theorem}

\nocite{*}


\end{document}